\documentclass[12pt]{article}
\usepackage{amsmath,amssymb}
\usepackage{amsthm}

\def\PC{\mathcal{P}}

\def\MC{\mathcal{M}}

\def\E{\mathbf{E}}

\def\R{\mathbf{R}}

\def\1{\mathbf{1}}

\def\al{\alpha}
\def\be{\beta}
\def\pa{\partial}
\def\ep{\epsilon}
\def\de{\delta}

\newtheorem{theorem}{Theorem}[section]

\newtheorem{remark}{Remark}

\newcommand{\la}{\lambda}
\newcommand{\si}{\sigma}

\newcommand{\Om}{\Omega}
\newcommand{\La}{\Lambda}

\begin{document}
\title{On the mean field games with common noise and the McKean-Vlasov SPDEs
\thanks{preprint}
}
\author{
Vassili N. Kolokoltsov\thanks{Department of Statistics, University of Warwick,
 Coventry CV4 7AL UK, and associate member  of IPI RAN RF, Email: v.kolokoltsov@warwick.ac.uk}
 and Marianna Troeva\thanks{Research Institute of Mathematics, North-Eastern Federal University,
58 Belinskogo str., Yakutsk 677000 Russia, Email: troeva@mail.ru
}}
\maketitle

\begin{abstract}
We formulate the MFG limit for $N$ interacting agents with a common noise as a single quasi-linear
deterministic infinite-dimensional partial differential second order backward equation.
We prove that any its (regular enough) solution provides an $1/N$-Nash-equilibrium profile
for the initial $N$-player game.
We use the method of stochastic characteristics to provide the link with the basic models
of MFG with a major player.
We develop two auxiliary theories of independent interest: sensitivity and regularity analysis
for the McKean-Vlasov SPDEs and the $1/N$-convergence rate for the propagation of chaos property of interacting diffusions.
\end{abstract}


\section{Introduction}

We shall denote by $\MC^{sign}(\R^n)$ the space of signed Borel measures on $\R^n$ of finite total variation,
by $\MC(\R^n)$ its cone of positive measures, by $\MC_{\la}^{sign}(\R^n)$, $\MC_{\la}(\R^n)$
their subsets of total variation norm not exceeding $\la$,
by $\PC(\R^n)$ the set of probability measures.
We shall use the standard notation $(\phi, \mu)=\int \phi (x) \mu (dx)$ for the pairing of functions and measures.
By $\E$ we denote the expectation.

Let us consider $N$ agents, whose positions are governed by the system of SDEs
\begin{equation}
\label{eqstartSDEN}
dX_t^i=b(t,X_t^i, \mu_t^N, u_t^i)\, dt +\si_{ind}(X_t^i) dB_t^i +
\si_{com}(X_t^i) dW_t,
\end{equation}
where all $X_t^i$ belong to $\R$, $W_t, B_t^1, \cdots, B_t^N$ are independent one-dimensional standard Brownian motions, $W_t$,
referred to as the common noise,
and all $B^j_t$, referred to as the idiosyncratic or individual noises,
the subscripts 'com' and 'ind' referred to the objects related to the common or to the individual noises.
The parameters $u_t^i\in U\subset \R^m$ are controls available to the players,
trying to minimize
their payoffs
\begin{equation}
\label{eqstartPayN}
V^i_{[t,T]}(x)=\E \left[ \int_t^T J(s,X_s^i, \mu_s^N, u_s^i) \, ds +V_T(X_T^i, \mu_T^N)\right],
\end{equation}
depending on the action of other players,
with the given functions $J$ and $V_T$.
The coefficient $b(t,x, \mu, u)$ is a function of $t\in \R, x\in \R, u\in U$ and
a measure $\mu \in \MC^{sign}(\R)$, and $\mu_t^N$ in
\eqref{eqstartSDEN} is
\[
\mu_t^N=\frac{1}{N}\sum_{i=1}^N \de_{X_t^i}.
\]

In general, the function $b$ needs to be defined only for $\mu$ from the set of probability measures
$\PC(R^n)$. However, to use smoothness with respect to $\mu$ it is convenient (though not necessary)
 to have this function defined on a larger space. In the usual examples, $b$
depend on $\mu$ via a finite set of moments of type
\begin{equation}
\label{eqdemomonmu}
F_j(\mu)= \int \tilde F_j(x_1, \cdots, x_k) \mu(dx_1) \cdots \mu(dx_k),
\end{equation}
with some bounded measurable symmetric functions $\tilde F_j$.

For simplicity, we shall assume $b$ to be linear in $u$, that is,
\begin{equation}
\label{eqdefdriftlinearincontrol}
b(t,x,\mu,u)= b_1(t,x,\mu)+b_2(t,x,\mu)u,
\end{equation}
though other weaker assumptions are possible.

\begin{remark}
1) For this paper we shall stick to a smooth dependence of $b$ on $\mu$.
However, more singular dependencies are also of interest, for instance, the dependence on $\mu$ via its
quantile, see \cite{CKL}. This case will be discussed in our subsequent publication based on the theory
of SDEs with coefficients depending on quantiles developed in \cite{Kol13}.
2) In an attempt to present our main result in the most clear way, we make several simplifying assumptions,
primarily that  all objects are one dimensional and that $\si_{ind}$ does not depend on $\mu$ and $u$,
which can be relaxed causing the increase of technicalities.
3) We consider the simplest common noise $W_t$. It would be natural to extend the theory
to the space-time white noise $W(dxdt)$ or even to a more general noise expressed in terms of
functional semimartingales $F(x,t)$ analyzed in \cite{K1997}.
\end{remark}

It is known (see e. g. \cite{KurXi99}) that, for fixed common functions
$u^i_t(X_t^i)=u_t(X_t^i)$, and under appropriate regularity assumptions on $b,\si_{ind},\si_{com}$
the system \eqref{eqstartSDEN} is well-posed and the corresponding empirical measures $\mu_t^N$
converge, as $N\to \infty$, to the unique solution $\mu_t$ of the nonlinear SPDE of the McKean-Vlasov type
\begin{equation}
\label{eqlimMacVlaSPDE}
d(\phi, \mu_t)=(L[t,\mu_t,u_t]\phi, \mu_t) \, dt +(\si_{com}(.) \nabla \phi, \mu_t) \, dW_t,
\end{equation}
which is written here in the weak form meaning that it should hold for all $\phi \in C^2(\R)$, and where
$\nabla $ is the derivative with respect to the space variable $x$
and
\begin{equation}
\label{eqlimMacVlaSPDE1}
L[t,\mu_t,u_t]\phi (x)= \frac12
(\si_{ind}^2+\si_{com}^2)(x)\frac{\pa^2 \phi}{\pa x^2}
+b(t,x, \mu_t, u_t(x))\frac{\pa \phi}{\pa x}.
\end{equation}

Let us mention directly that in our approach it is more convenient to work with the Stratonovich
differentials. Namely, by the usual rule $Y\circ dX=Y dX +\frac12 dY dX$, equation \eqref{eqlimMacVlaSPDE}
rewrites in the Stratonovich form as
\begin{equation}
\label{eqlimMacVlaSPDEst}
d(\phi, \mu_t)=(L_{St}[t,\mu_t,u_t]\phi, \mu_t) \, dt +(\si_{com}(.) \nabla \phi, \mu_t) \circ dW_t,
\end{equation}
with
\begin{equation}
\label{eqlimMacVlaSPDE1st}
L_{St}[t,\mu_t,u_t]\phi (x)= \frac12
\si_{ind}^2(x)\frac{\pa^2 \phi}{\pa x^2}
+[b(t,x, \mu_t, u_t(x))-\frac12 \si_{com} \si'_{com}(x)]\frac{\pa \phi}{\pa x}.
\end{equation}

With some abuse of notation,
we shall often identify measures with their densities (whenever they exist) with respect to Lebesgue measures,
thus writing the strong form of equation \eqref{eqlimMacVlaSPDE} as
\begin{equation}
\label{eqlimMacVlaSPDEstr}
d\mu_t=L'[t,\mu_t,u_t] \mu_t \, dt -\nabla (\si_{com}(.) \mu_t) \, dW_t,
\end{equation}
with
\begin{equation}
\label{eqlimMacVlaSPDE1str}
L'[t,\mu,u_t]\nu = \frac12
\frac{\pa^2}{\pa x^2} \left[(\si_{ind}^2+\si_{com}^2) \mu)\right]
-\frac{\pa }{\pa x} [b(t,x, \mu, u_t)\mu].
\end{equation}
This identification does not cause ambiguity, because
under non-degeneracy of $\si_{ind}^2+\si_{com}^2$ that we shall always assume,
any solution $\mu_t[\mu_0]$ to  \eqref{eqlimMacVlaSPDE} has a density with respect to Lebesgue measure
at any $t>0$, even if $\mu_0$ does not.

Recall now that the optimal control problem facing each player, say $X^1_t$,
is to minimize cost \eqref{eqstartPayN}.
Now the crucial difference with the games without common noise starts to reveal itself. For games without noise, one expects
to get a deterministic curve $\mu_t$ in the limit of large $N$, so that in the limit, each player
should solve a usual optimization problem for a diffusion in $\R$. Here the limit is stochastic, and thus even in the limit
the optimization problem faced by each player is an optimization with respect to an infinite-dimensional,
in fact measure-valued, process.

In fact, for fixed $N$, if all players, apart from the first one, are using the same control $u_{com}(t,x,\mu)$,
 the optimal payoff for the first player is found from the HJB equation for the diffusion
governed by \eqref{eqstartSDEN}, that is, the HJB equation (where we denote $X^1$ by $x$),
\[
\frac{\pa V}{\pa t}(x, \mu)+\inf_{u} \left[b(t,x, \mu_t, u) \frac{\pa V}{\pa x}+J(t,x, \mu_t, u)\right]
+\frac12
(\si^2_{ind}+\si^2_{com})(x)\frac{\pa^2 V}{\pa x^2}
\]
\[
+\sum_{j\neq 1} b(t,x_j, \mu_t, u(t,x_j,\mu)) \frac{\pa V}{\pa x_j}+\frac12
(\si^2_{ind}+\si^2_{com})(x_j)\frac{\pa^2 V}{\pa x_j^2}
\]
\begin{equation}
\label{eqHJBcomnoise1}
+\sum_{j\neq 1} \si_{com}(x)\si_{com}(x_j)\frac{\pa^2 V}{\pa x_1 \pa x_j}
+\sum_{1<i<j}\si_{com}(x_i)\si_{com}(x_j)\frac{\pa^2 V}{\pa x_i \pa x_j}=0.
\end{equation}

As will be shown, in the limit when $(\de_{x_1}+\cdots +\de_{x_N})/N$ converge to the process $\mu_t$, this equation
turns to the limiting HJB equation
\[
\frac{\pa V}{\pa t}(x, \mu)+\inf_{u} \left[b(t,x, \mu, u) \frac{\pa V}{\pa x}+J(t,x, \mu, u)\right]
+\frac12
(\si^2_{ind}+\si^2_{com})(x)\frac{\pa^2 V}{\pa x^2}
\]
\begin{equation}
\label{eqHJBcomnoise2}
+\La_{lim} V(x,\mu) +\int \si_{com}(x)\si_{com}(y)\frac{\pa^2}{\pa x \pa y}\frac{\de V (x,\mu)}{\de \mu(y)} \mu(dy)=0,
\end{equation}
where the operator $\La_{lim}$ is calculated in \eqref{eqstartSDENgenonmes1} with $u_{com}$ as the control.

If $J$ is convex, the infimum here is achieved on the single point
\begin{equation}
\begin{split}
\label{eqsfixedmuHJBcont}
\hat u_{ind}(t,x,\mu)&=argmax_{u} \left[b(t,x, \mu_t, u) \frac{\pa V}{\pa x}+J(t,x, \mu_t, u)\right]\\
&=-\left(\frac{\pa J}{\pa u}\right)^{-1}\left(b_2(t,x,\mu_t)\frac{\pa V}{\pa x}\right).
\end{split}
\end{equation}

Now the difference with the usual MFG is fully seen. Instead of a pair of coupled forward-backward equations
we have now one single infinite-dimensional equation \eqref{eqHJBcomnoise2}.
Namely, for any curve $u_{com}(t,x,\mu)$ (defining $\La_{lim}$ in  \eqref{eqstartSDENgenonmes1} and thus in
\eqref{eqHJBcomnoise2}), we should solve equation \eqref{eqHJBcomnoise2} with a given terminal condition
leading to the optimal control \eqref{eqsfixedmuHJBcont}. The key {\it MFG consistency requirement} is
now given by the equation
\begin{equation}
\label{eqMFGconsistcomnoise}
\hat u_{ind}(t,x,\mu)=u_{com}(t,x,\mu).
\end{equation}

This can be interpreted as having a limiting game of two players, a tagged player
and a measure-valued player, for which we are looking for a symmetric Nash equilbrium.

Equivalently, the MFG consistency \eqref{eqMFGconsistcomnoise} can be encoded into a single quasi-linear
 deterministic infinite-dimensional partial differential second order backward equation on the function $V(t,x,\mu)$,
  which we present now in full
 substituting  $\La_{lim}$ from  \eqref{eqstartSDENgenonmes1} and \eqref{eqMFGconsistcomnoise} in
 \eqref{eqHJBcomnoise2}:
\[
\frac{\pa V}{\pa t}(t,x\mu)+\left.\left[b(t,x, \mu, u) \frac{\pa V}{\pa x}+J(t,x, \mu, u)\right]
\right|_{u(t,x,\mu)=-(\pa J/\pa u)^{-1}(b_2(t,x,\mu)\frac{\pa V}{\pa x})}
\]
\[
+\frac12
(\si^2_{ind}+\si^2_{com})(x)\frac{\pa^2 V}{\pa x^2}
+\frac12 \int_{\R^2}\!\!\si_{com}(y)\si_{com}(z)\frac{\pa^2}{\pa y \pa z}\frac{\de^2 V (t,x,\mu)}{\de \mu(y)\de \mu(z)} \mu(dy) \mu (dz)
\]
\[
+\int\left(\left[\left.b(t,.,\mu,u(t,.,\mu))\right|_{u(t,z,\mu)=-(\frac{\pa J}{\pa u})^{-1}(b_2(t,z,\mu)\frac{\pa V}{\pa z})}\nabla
+\frac12 (\si^2_{ind}+\si^2_{com})(.)\nabla^2\right]\right.
\]
\begin{equation}
\label{eqMFGconsistcomnoise1}
\left.\cdot\frac{\de V(t,x,\mu)}{\de \mu (.)}\right)(y) \mu (dy)+\int \si_{com}(x)\si_{com}(y)\frac{\pa^2}{\pa x \pa y}\frac{\de V (x,\mu)}{\de \mu(y)} \mu(dy)=0,
\end{equation}
with a given terminal condition
\begin{equation}
\label{eqsforbackboundary}
V_t(x,\mu)|_{t=T}=V_T(x,\mu), \quad \mu_t|_{t=0}=\mu_0.
\end{equation}

If
\[
J(t,x,\mu,u)=\frac12 u^2, \quad b_2(t,x,\mu)=1,
\]
then \eqref{eqMFGconsistcomnoise1} simplifies to
\[
 \frac{\pa V}{\pa t}(t,x,\mu)+\left[b\left(t,x, \mu,  -\frac{\pa V}{\pa x}\right) \frac{\pa V}{\pa x}+\frac12 \left(\frac{\pa V}{\pa x}\right)^2\right]
 \]
 \[
+\frac12
(\si^2_{ind}+\si^2_{com})(x)\frac{\pa^2 V}{\pa x^2}
+\frac12 \int_{\R^2} \si_{com}(y)\si_{com}(z)\frac{\pa^2}{\pa y \pa z}\frac{\de^2 V (t,x,\mu)}{\de \mu(y)\de \mu(z)} \mu(dy) \mu (dz)
\]
\[
+\int \left(\left[ b(t,.,\mu, -\nabla V(t,.,\mu))\nabla
 +\frac12 (\si^2_{ind}+\si^2_{com})(.)\nabla^2\right] \frac{\de V(t,x,\mu)}{\de \mu (.)}\right)(y) \mu (dy)
\]
 \begin{equation}
\label{eqMFGconsistcomnoise1sim}
 +\int \si_{com}(x)\si_{com}(y)\frac{\pa^2}{\pa x \pa y}\frac{\de V (x,\mu)}{\de \mu(y)} \mu(dy)=0.
\end{equation}

The MFG methodology suggests that
for large $N$ the optimal behavior of players arises from the control $\hat u$ given by \eqref{eqsfixedmuHJBcont}
with $V$ solving  \eqref{eqMFGconsistcomnoise1},
or equivalently, satisfying the consistency condition \eqref{eqMFGconsistcomnoise}.

To justify this claim one is confronted essentially with the 3 problems:

MFG1): Prove well-posedness of (or at least the existence of the solution to) the problem
\eqref{eqMFGconsistcomnoise1} or \eqref{eqMFGconsistcomnoise};

MFG2): Analyze the Nash equilibria of the $N$-player game given by
\eqref{eqstartSDEN}, \eqref{eqstartPayN} and prove that these equilibria
(or at least their subsequence) converge, as $N\to \infty$, to a solution
of the problem \eqref{eqMFGconsistcomnoise1} or \eqref{eqMFGconsistcomnoise};
assess the convergence rates;

MFG3): Show that a solution to the problem \eqref{eqMFGconsistcomnoise1} or \eqref{eqMFGconsistcomnoise}
provides a profile of symmetric strategies $\hat u_t(x)$, which is an $\ep$-Nash equilibrium
of the $N$-player game given by
\eqref{eqstartSDEN}, \eqref{eqstartPayN} and the initial distribution of players $\mu_0$,
with $\ep (N)\to 0$ as $N\to \infty$; estimate the error-term $\ep(N)$.

Questions MFG2), MFG3) are of course two facets of the same problem on how well
the solutions to the limiting problem
(\eqref{eqlimMacVlaSPDE},\eqref{eqHJBcomnoise2},\eqref{eqsfixedmuHJBcont},\eqref{eqsforbackboundary})
 approximate a finite player game, but the methods of dealing with these problems can be rather different.

To link with the usual MFG, let us notice that for the case without common noise given by
\eqref{eqstartSDEN} with $\si_{com}=0$, equation \eqref{eqMFGconsistcomnoise1sim}, say, turns to
\[
 \frac{\pa V}{\pa t}(t,x,\mu)+\left[b\left(t,x, \mu,  -\frac{\pa V}{\pa x}\right) \frac{\pa V}{\pa x}+\frac12 \left(\frac{\pa V}{\pa x}\right)^2\right]
+\frac12
\si^2_{ind}(x)\frac{\pa^2 V}{\pa x^2}
\]
 \begin{equation}
\label{eqMFGconsistcomnoise1sim}
+\int \left(\left[ b(t,.,\mu, -\nabla V(t,.,\mu))\nabla
 +\frac12 \si^2_{ind}(.)\nabla^2\right] \frac{\de V(t,x,\mu)}{\de \mu (.)}\right)(y) \mu (dy)=0,
\end{equation}
giving a single-equation approach to usual MFG. In fact, solving this equation for a function $V(t,x,\mu)$ is equivalent
to solving first the deterministic (forward) equation \eqref{eqlimMacVlaSPDEstr} with $\si_{com}=0$
and then the backward equation
\[
 \frac{\pa V}{\pa t}(t,x,\mu_t)+\left[b\left(t,x, \mu_t, -\frac{\pa V}{\pa x}\right) \frac{\pa V}{\pa x}+\frac12 \left(\frac{\pa V}{\pa x}\right)^2\right]
+\frac12
\si^2_{ind}(x)\frac{\pa^2 V}{\pa x^2}=0
\]
for a function $V(t,x)$.

In a more abstract form the link between the forward-backward formulation and the single backward formulation is as follows.
If $(x_t,\mu_t)$ is a controlled Markov process (not necessarily measure-valued),
optimal payoff is defined via the corresponding HJB on a function $V(t,x,\mu)$ of three arguments (corresponds to our general common noise case).
If the evolution of the coordinate $\mu_t$ is deterministic and does not depend on $x$ and its control,
one can (alternatively and equivalently) first solve this deterministic equation on $\mu$
(usual forward part of the basic MFG) and then substitute it in the basic HJB to get the
equation on $V(t,x)$, the function of two arguments only,
with $\mu_t$ included in the time dependence (usual backward part of the basic MFG).
This decomposition into forward-backward system is not available in general.

In this paper we are going to concentrate exclusively on question MFG3), aiming at proving the error-estimate
of order $\ep(N)\sim 1/N$. Our approach will be based on interpreting
(by means of Ito's formula) the common noise as a kind of binary interaction of agents
(in addition to the usual mean-field interaction of the standard situation without common noise)
and then reducing the problem to the sensitivity analysis for McKean-Vlasov SPDE.

The question MFG1) can be approached via the methods of papers
\cite{KoWe13}, \cite{KoWe14}, which will be addressed in another publication.
Some existence can be also derived from \cite{CarDelLac14}, which has however a slightly different formulation than the present one.

Our paper is organized as follows.
Next section provides a short literature review.
Section \ref{secstrat} formulates our main results and indicates the strategy of their proof.
Sections \ref{secreg}-\ref{secMar}
are devoted to the regularity and sensitivity analysis of the solutions to the
McKean-Vlasov SPDEs and the related properties of the corresponding measure-valued Markov processes.
The last three Sections prove the Theorems formulated in Section \ref{secstrat}.

\section{Brief literature review}

Mean-field games present a quickly developing area of the game theory.
It was initiated by Lasry-Lions \cite{LL2006} and Huang-Malhame-Caines
\cite{HCM3}, \cite{HCM07}, \cite{Hu10},
see \cite{Ba13}, \cite{Ben13}, \cite{GLL2010}, \cite{Gomsurv}, \cite{Cain14} for recent surveys,
as well as \cite{Car15}, \cite{Car13}, \cite{Gom14} and references therein.

New trends concern the theory of mean-field games with a major player,
see \cite{Nour13}, the numeric analysis, see \cite{Achd13},  and the games with a discrete state space,
see \cite{Gom14a} and references therein.

Even more recent development deals with
mean-field games with common noise, which are only starting to be analyzed.
Of course, common noise can be considered as a kind of neutral major player,
but the usual setting for the latter \cite{Nour13} introduces the corresponding noise
into the coefficients of the SDEs of the minor players, rather than adding additional
common stochastic differential. One of the ideas (and results) of our contribution is
to use the method of stochastic characteristics to link these two models.

Some simple concrete models of mean-field game types with common noise
applied to modeling inter bank loans are analyzed in detail
in \cite{Car15a}.
A model of common noise with constant coefficients is discussed in \cite{Ahuja}.
Seemingly first serious contributions to the general theory of mean-field games with common noise are the preprints
\cite{CarDelLac14} and \cite{Lack14}, which includes well-posedness for the mean-field limiting evolution
under certain assumptions.  However, \cite{CarDelLac14} and \cite{Lack14} work mostly
with controlled SDEs, and our approach is rather different, being based on McKean-Vlasov SPDEs.
The references on the literature on McKean-Vlasov equation are given in the Sections devoted
to this equations.

\section{Our strategy and results}
\label{secstrat}

Our main result is the following.

\begin{theorem}
\label{th1}
Let $b, \si_{com}, \si_{ind}$ satisfy the assumption of Theorem \ref{th2} below
 and let $V(t,x,\mu)$ be a solution to problem
\eqref{eqMFGconsistcomnoise1},\eqref{eqsforbackboundary}.
Assume $J(t,x,\mu,u(t,x,\mu))$ and $V(t,x,\mu)$, as functions of $(x,\mu)$ satisfy the assumptions on
function $F$ from Theorem \ref{th3} below.
Then the profile of symmetric strategies $\hat u_t(x, \mu)$ given by
\eqref{eqsfixedmuHJBcont} is an $\ep$-Nash equilibrium
of the $N$-player game given by
\eqref{eqstartSDEN}, \eqref{eqstartPayN}, with $\ep (N)\sim 1/N$ as $N\to \infty$.
\end{theorem}

\begin{remark} The assumptions can be weaken in many ways, but some regularity of the control
synthesis $u$ (like being Lipschitz in $x,\mu$) is definitely needed for the rather subtle
estimate $1/N$.
\end{remark}

Additionally,
in preparation to this result, we obtain two other results of independent interest,
not linked with any optimization problem, namely the regularity and sensitivity
for McKean-Vlasov SPDE, Theorems \ref{thwellposMcVlSPDE2} and \ref{thwellposMcVlSPDE3},
and the $1/N$-rates of convergence for interacting diffusions
to the limiting measure-valued diffusion, Theorem \ref{th2} (often interpreted as the 'propagation of chaos' property).
Notice that the convergence itself is a known result (see e.g. \cite{DaVa95} or \cite{KurXi99}).
The well-posedness of the McKean-Vlasov SPDE was shown in \cite{KurXi99} in the class
of $L_2$-functions, and for measures in \cite{DaVa95}, though under an additional monotonicity assumption.

Let us fix some basic notations for the function spaces.
For a topological space $X$, $C(X)$ denotes the Banach space of continuous functions equipped with the sup-norm $\|.\|$.
The topology on measures will be always the weak one.
If $X=\R^d$, then $C^k(X)$  denotes the Banach space of functions with all derivatives up to order $k$ belonging to $C(X)$,
 $L_1(X)$ denotes the space of integrable functions, $L_{\infty}(X)$ the space of bounded measurable functions
 with the essential supremum as a norm, $H^1_1(X)$ the Sobolev space of integrable functions such that its generalized
  derivative is also integrable. If $X$ is not indicated explicitly
   in this notations we mean $X=\R$.

Let $C^{k\times k}(\R^{2d})$ denote the space of functions $f$ on $\R^{2d}$ such that the partial derivatives
\[
\frac{\partial ^{\alpha+\beta} f}{\partial x^\alpha \partial y^\beta}
\quad \text{with multi-index } \alpha, \beta \text{ such that }|\alpha|\leq k, |\beta|\leq k,
\]
belong to $C(\R^{2d})$.

\begin{remark}
The space $C^{k\times k}(\R^{2d})$ looks a bit exotic. However, it is very natural
 for the study of the second order derivatives of nonlinear measure-valued flows.
The spaces of this kind also play an important role in the analysis of stochastic flows in H. Kunita \cite {K1997},
though Kunita's spaces are slightly more general as they allow for a linear growth of functions.
\end{remark}

Recall that for a functional $F(\mu)$ on $\MC^{sign}(\R^d)$, the variational derivative is defined as
\[
\frac{\de F}{\de \mu(x)}[\mu]=\frac{d}{dh}|_{h=0}F(\mu +h\de x).
\]
Derivatives of higher order are defined accordingly.
For instance, if $F=F_j$ is given by \eqref{eqdemomonmu}, then
\[
\frac{\de F}{\de \mu(x)}= k \int \tilde F_j(x, x_2, \cdots, x_k) \mu(dx_2) \cdots \mu(dx_k),
\]
\[
\frac{\de ^2 F}{\de \mu(x) \de \mu(y) }= k(k-1) \int \tilde F_j(x, y, x_3, \cdots, x_k) \mu(dx_3) \cdots \mu(dx_k).
\]


Let $C^k(\MC_{\la}^{sign}(\R^d))$ denote the space of functionals such that the $k$th order variational derivatives are well defined
and represent continuous functions. It is a Banach space with the norm
\[
\|F\|_{C^k(\MC_{\la}^{sign}(\R^d))}
=\sum_{j=0}^k \sup_{x_1, \cdots , x_j, \mu\in \MC_{\la}^{sign}(\R^d)} \left|\frac{\de^j F}{\de \mu(x_1)\cdots\de \mu(x_j)}\right|.
\]
Let $C^{k,l}(\MC_{\la}^{sign}(\R^d))$ denote the subspace of $C^k(\MC_{\la}^{sign}(\R^d))$ such that all derivatives up to order
$k$ have continuous bounded derivatives up to order $l$ as functions of their spatial variables. It is a Banach space with the norm
\[
\|F\|_{C^{k,l}(\MC_{\la}^{sign}(\R^d))}=\sum_{j=0}^k \sup_{\mu\in \MC_{\la}^{sign}(\R^d)}
\left\|\frac{\de^j F}{\de \mu(.)\cdots\de \mu(.)}[\mu]\right\|_{C^l(\R^{dj})}.
\]
Finally, let $C^{2,k\times k}(\MC_{\la}^{sign}(\R^d))$ be the space of functionals with the norm
\[
\|F\|_{C^{2,k\times k}(\MC_{\la}^{sign}(\R^d))}
=\sup_{\mu\in \MC_{\la}^{sign}(\R^d)}\left \|\frac{\de^2 F}{\de \mu (.) \de \mu(.)}\right\|_{C^{k \times k}(\R^{2d})}.
\]
These Banach spaces are natural objects for studying sensitivity for nonlinear measure-valued evolutions.
As we are interested mostly in probability measures, we shall usually tacitly assume $\la=1$ for these spaces.

As the derivatives of measures are not always measures (say, the derivative of $\de_x$ is $\de'_x$), to study the derivatives of
the nonlinear evolutions one needs the spaces dual to the spaces of smooth functions. Namely, for a generalized function
(distribution) $\xi$ on $\R^d$ we say that it belongs to the space $[C^k(\R^d)]'$ if the norm
\[
\|\xi\|_{[C^k(\R^d)]'}=\sup_{\phi: \|\phi\|_{C^k(\R^d)}\le 1} |(\xi, \phi)|
\]
is finite. For instance,
\[
\|\de^{(k)}_x \|_{[C^k(\R^d)]'}=1.
\]
We shall use these norms mostly for generalized functions that are given by locally integrable functions.
In this case the $[C(\R^d)]'$-norm coincides with the $L_1$ norm. To see why these spaces are handy,
 let us observe that if we take a spatial derivative of a heat kernel, then its
 $L_1$-norm is of order $t^{-1/2}$ for small $t$, but its $[C'(\R^d)]'$-norm is uniformly bounded.

Let us explain our strategy for proving Theorem \ref{th1}.

For any $N$ and a fixed common strategy $u_t(x,\mu)$,  solutions to the system of SDEs \eqref{eqstartSDEN}
on $t\in [0,T]$
define a backward propagator (also referred in the literature as a flow or as a two-parameter semigroup)
$U^{s,t}_N=U^{s,t}_N[u(.)]$, $0\le s\le t \le T$, of linear contractions on the space $C_{sym}(\R^N)$ of symmetric functions
via the formula
\begin{equation}
\label{eqbackpropNpart}
(U^{s,t}_Nf)(x_1, \cdots , x_N)=\E f (X_1, \cdots , X_N)_{s,t}(x_1, \cdots , x_N),
\end{equation}
where $(X_1, \cdots , X_N)_{s,t}(x_1, \cdots , x_N)$ is the solution to
\eqref{eqstartSDEN} at time $t$ with the initial condition
\[
(X_1, \cdots , X_N)_{s,s}(x_1, \cdots , x_N)=(x_1, \cdots , x_N)
\]
at time $s$. The corresponding dual forward propagator $V_N^{t,s} =(U_N^{s,t})'$
is defined by the equation
\begin{equation}
\label{eqforpropNpart}
(f, V_N^{t,s}\mu)=(U^{s,t}_N f, \mu).
\end{equation}
It acts
on the probability measures on $\R^N$, so that if $\mu$ is the initial distribution
of $(X_1, \cdots, X_N)$ at time $s$, then $V_N^{t,s}\mu$ is the distribution
of  $(X_1, \cdots, X_N)$ at time $t$.

By the standard inclusion
\begin{equation}
\label{eqstandinclusionpartmeasure}
(x_1, \cdots , x_N) \to \frac{1}{N} (\de_{x_1}+\cdots +\de_{X_N})
\end{equation}
the set $\R^N$ is mapped to the set $\PC_N(\R)$ of normalized sums of $N$ Dirac's measures,
so that $U^{s,t}_N$, $V^{t,s}_N$ can be considered as propagators in $C(\PC_N(\R))$ and
$\PC(\PC_N(\R))$ respectively.

On the other hand, for  a fixed function $u_t(x,\mu)$, the solution of SPDE \eqref{eqlimMacVlaSPDE}
specifies a stochastic process, a diffusion, on the space of probability measures $\PC(\R)$
defining the backward propagator $U^{s,t}=U^{s,t}[u(.)]$ on $C(\PC(\R))$:
\begin{equation}
\label{eqbackproplim}
(U^{s,t}f)(\mu)=\E f (\mu_{s,t}(\mu)),
\end{equation}
where $\mu_{s,t}(\mu)$ is the solution to \eqref{eqlimMacVlaSPDE} at time $t$ with a given initial condition
$\mu$ at time $s\le t$.

From the convergence of the empirical measures $\mu_t^N$, mentioned above, it follows
that   $U^{s,t}_N$ tend $U^{s,t}$, as $N\to \infty$. The following result provides the rates for the weak convergence.

\begin{theorem}
\label{th2}
Assume $\si_{ind}, \si_{com} \in C^3(\R)$ and are positive functions never approaching zero. Assume
\[
b(t,x,.,u(t,x,.))\in (C^{2,1\times 1}\cap C^{1,2})(\MC_1^{sign}(\R)),
\]
\[
b(t,.,\mu, u(t,.,\mu)) \in C^2(\R), \quad \frac{\pa b}{\pa x}(t,x,., u(t,x,.))\in  C^{1,0}(\MC_1^{sign}(\R))
\]
with bounds uniform with respect to all variables.
Then, for any $\mu \in \PC_N(\R)$ and $F\in (C^{2,1\times 1}\cap C^{1,2})(\MC_1^{sign}(\R))$
\begin{equation}
\label{eq1th2}
\|(U^{s,t}-U^{s,t}_N)F(\mu)\|_{C(\MC_1^{sign}(\R))}
\le \frac{C(T)}{N} \left(\|F\|_{C^{2,1\times 1}(\MC_1^{sign}(\R))}
+\|F\|_{C^{1,2}(\MC_1^{sign}(\R))}\right)
\end{equation}
for $0\le s \le t \le T$.
\end{theorem}


This result belongs to the statistical mechanics of interacting diffusions,
 so that its significance goes beyond any
links with games or control theory.

This result is not sufficient for us, as we have to allow one of the agent to behave differently
from the others. To tackle this case we shall considered the corresponding problem with a tagged agent.
Namely, consider the Markov process on pairs $(X_t^{1,N}, \mu_t^N) [u^{ind}(.), u^{com}(.)]$,
where $u^{ind}$ and $u^{com}$ are some $U$-valued functions $u_t^{ind}(x, \mu)$, $u_t^{com}(x,\mu)$,
$(X_t^{1,N}, \cdots , X_t^{N,N})$ solves \eqref{eqstartSDEN} under the assumptions that
the first agent uses the control $u_t^{ind}(X_t^{1,N}, \mu_t^N)$ and all other agents $i\neq 1 $ use the control
$u_t^{com}(X_t^{i,N}, \mu_t^N)$, and
$\mu_t^N=\frac{1}{N}\sum_{i=1}^N \de_{X_t^{i,N}}$.

\begin{remark}
The coordinates $(X_t^{1,N}, \mu_t^N)$ of our pair process are not independent.
Quite opposite, $X_t^{1,N}$ is the position of the first $\de$-function in $\mu_t^N$. However,
we are aiming at the limit $N\to \infty$ where the influence of $X_t^{1,N}$
on $\mu_t^N$ becomes negligible, and we do not want it to be lost in the limit.
Alternatively, to avoid this dependence, one can consider (as some authors do), instead of our $\mu_t^N$,
 the measures that do not take $X_t^{1,N}$ into account, that is $\tilde \mu_t^N=\frac{1}{N}\sum_{i=2}^N \de_{X_t^{i,N}}$,
 but this would neither change the results, nor simplifies the notations.
 \end{remark}

Let us now define the corresponding tagged propagators
$U^{s,t}_{N,tag}=U^{s,t}_{N,tag}[u^{ind}(.), u^{com}(.)]$ and $U^{s,t}_{tag}=U^{s,t}_{tag}[u^{ind}(.), u^{com}(.)]$:
\begin{equation}
\label{eqbackpropNparttag}
(U^{s,t}_{N,tag}F)(x, \mu)=\E F (X_t^{1,N}, \mu_t^N) [u^{ind}(.), u^{com}(.)](x,\mu),
\end{equation}
where $\mu=\frac{1}{N}\sum_{j=1}^N\de_{x_j}$ is the position of the process at time $s$ and where $x=x_1$;
\begin{equation}
\label{eqbackproplimtag}
(U^{s,t}_{tag}F)(x, \mu)=\E F (X_t^1, \mu_t) [u^{ind}(.), u^{com}(.)](x,\mu),
\end{equation}
where the process $(X_t^1, \mu_t)[u^{ind}(.), u^{com}(.)](x,\mu)$ with the initial data $x,\mu$ at time $s$
is the solution to the system of stochastic equations
\begin{equation}
\label{eqstartSDENtag}
dX_t^1=b(t,X_t^1, \mu_t, u_t^{ind}(X_t^1,\mu_t)) +\si_{ind}(X_t^1) dB_t^1 +
\si_{com}(X_t^1) dW_t,
\end{equation}
\begin{equation}
\label{eqlimMacVlaSPDEtag}
d(\phi, \mu_t)=(L[t,\mu_t,u_t^{com}(.,\mu_t)]\phi, \mu_t) \, dt +(\si_{com}(.) \nabla \phi, \mu_t) \, dW_t
\end{equation}
(the second equation is actually independent of the first one).

The following is the basic convergence result for the tagged processes.

\begin{theorem}
\label{th3}
Under the assumptions of Theorem \ref{th2} (with both $u_t^{com}, u_t^{ind}$ satisfying these assumptions),
let $F(x,\mu)$, $x\in \R$, belongs to the space
$(C^{2,1\times 1}\cap C^{1,2})(\MC_1^{sign}(\R))$ as a function of $\mu$, $F\in C^2(\R)$ as a function
of $x$ and $\frac{\pa F}{\pa x}(x,.)\in C^{1,1}(\MC_1^{sign}(\R))$.
Then, for any $\mu \in \PC_N(\R)$
\begin{equation}
\label{eq1th2}
\|(U_{tag}^{s,t}-U^{s,t}_{N,tag})F\|_{C(\R\times \MC_1^{sign}(\R))}\le \frac{C(T)}{N}.
\end{equation}
\end{theorem}

Theorem \ref{th2} is a particular case of Theorem \ref{th3} obtained from the latter by ignoring the first coordinate.
However, by methodological reason, we first prove simpler Theorem \ref{th2}, and then its extension Theorem \ref{th3}.

It is now seen how Theorem \ref{th1} should be deduced. Since the evolutions $U_{tag}^{s,t}$ and $U^{s,t}_{N,tag}$
are close to each other, the corresponding optimal policies of the tagged agent should also be close.
Details are given in Section \ref{secth1}.


\section{On the regularity of McKean-Vlasov SPDEs}
\label{secreg}

In this and the next sections we develop the sensitivity analysis for McKean-Vlasov SPDEs,
which, on the one hand side, represent an important ingredient in the proof of our main result on mean-field games,
but on the other hand, has an independent significance for the theory of SPDEs.
Notice that there is quite an extensive literature on the properties of equation \eqref{eqstochMcKeanVlonedim}
with $A=0$ (see e. g. \cite{HuNu13}, \cite{CoJoKh13} and references therein),
but for $A=0$ much less is known, so that even the regularity
results from Theorem \ref{thwellposMcVlSPDE2} below seem to be new.

For a function $v(t,x)$, $t\ge 0, x\in \R$, let us consider the stochastic equation
\begin{equation}
\label{eqstochMcKeanVlonedim}
dv = L_t(v) \, dt +\Om v \circ dW_t,
\end{equation}
where
$W_t$ is a one-dimensional Brownian motion,
\begin{equation}
\label{eqstochMcKeanVlonedim1}
\Om v(x) =A(x) \frac{\pa v}{\pa x} +B(x) v (x),
\end{equation}
\begin{equation}
\label{eqstochMcKeanVlonedim2}
L_t(v) =\frac12 \si^2(x)\frac{\pa^2 v}{\pa x^2}+ b(t,x,[v]) \frac{\pa v}{\pa x} + c(t,x,[v]) v,
\end{equation}
with some functions $A(x),B(x),\si(x)$ and the functions $b,c$ depending in a smooth way on the function
(or a measure) $v$. To visualize this dependence, one can think of $b,c$ depending on $v$
 via a finite set of moments of type
\begin{equation}
\label{eqdemomonmurep}
F_j[v]= \int \tilde F_j(x_1, \cdots, x_k) v(x_1) \cdots v(x_k) \, dx_1 \cdots dx_k,
\end{equation}
with some bounded symmetric measurable functions $\tilde F_j$.

In \eqref{eqstochMcKeanVlonedim},  $\circ$ denotes the Stratonovich differential.
From the usual rule $Y\circ dX=Y \, dX +\frac12 dY \, dX$, one can rewrite
\eqref{eqstochMcKeanVlonedim} as an equation with Ito's differential of the similar kind:
\begin{equation}
\label{eqstochMcKeanVlonedim3a}
dv = L_t(v) \, dt +\Om v \, dW_t +\frac12 \Om^2 v \, dt,
\end{equation}
or explicitly
\begin{equation}
\label{eqstochMcKeanVlonedim3}
dv = L_t(v) \, dt +\Om v \, dW_t +\frac12\left[A^2 \frac{\pa^2 v}{\pa x^2}+A(2B +A')\frac{\pa v}{\pa x} +(AB'+B^2)v
\right] dt.
\end{equation}

Our objective is to study the well-posedness of equation \eqref{eqstochMcKeanVlonedim} and more importantly
its sensitivity with respect to initial conditions.

Our main assumptions will be that
\begin{equation}
\label{eqass1}
0<\si_1 \le \si(x) \le \si_2, \quad 0<A_1 \le A(x) \le A_2,
\end{equation}
that $\si \in C^2(\R), B\in C^2(\R), A\in C^3(\R)$ so that

\begin{equation}
\label{eqass2}
\sup_x \max(|\si'(x)|, |\si''(x)|, |B(x)|, |B'(x)|, |B''(x)|, |A'(x)|, | A''(x)|,|A'''(x)|)\le B_1,
\end{equation}

and
\begin{equation}
\label{eqass3}
\max \left( \| b(t,.,[v])\|_{C^1(\R)}, \|c(t,.,[v])\|_{C(\R)}\right) \le  b_1,
\end{equation}
\begin{equation}
\label{eqass4}
\sup_{t,y}\!\!\sup_{\|v\|_{\MC(\R)}\le \la}\!\!\max\left(\left\|\frac{\de b(t,y,[v])}{\de v(.)}\right\|_{C^1(\R)}\!\!,
 \left\|\frac{\de c(t,y,[v])}{\de v(.)}\right\|_{C(\R)}\!\!,
\left\|\frac{\de}{\de v(.)}\frac{\pa  b(t,y,[v])}{\pa y}\right\|_{C(\R)}\right)\!\!\le C(\la),
\end{equation}
with some constants $\si_1, \si_2, A_1, A_2, B_1, b_1$ and a function $C(\la)$.


As above, we shall often omit the arguments of various functions. Moreover sometimes, we shall write $v(t,x)$
 as $v_t(x)$ when stressing that certain operator acts on $v$ as a function of $x$ for a given $t$.

Our basic approach will be the method of stochastic characteristics, see \cite{K1997}, \cite{KoTyu03}, though
in its simplest form, available for one-dimensional noise. This method allows one to
turn equation  \eqref{eqstochMcKeanVlonedim} into a non-stochastic equation of the second order, but with random
coefficients. Namely, for $A(x),B(x)\in C^1(\R)$, operator \eqref{eqstochMcKeanVlonedim1} generates a contraction
group $e^{t\Om}$ in $C(\R)$, so that $e^{t\Om}v_0(x)$ is the unique solution to the equation
\[
\frac{\pa v}{\pa t}=\Om v
\]
with the initial condition $v(0,x)=v_0(x)$. Explicitly,
\begin{equation}
\label{eqmethcharac}
e^{t\Om}v_0(x)=v_0(Y(t,x))G(t,x), \quad t \in \R,
\end{equation}
where $Y(t,x)$ is the unique solution to the ODE $\dot Y=-A(Y)$ with the initial condition $Y(0,x)=x$ and
\[
G(t,x)=\exp \{ \int_0^t B(Y(s,x)) \, ds \}.
\]
In particular, $G$ has the properties:
\[
G(-t,x)=G^{-1} (t, Y(-t,x))=\exp \{ \int_0^{-t} B(Y(s,x)) \, ds \}=\exp \{ -\int_{-t}^0 B(Y(s,x)) \, ds \},
\]
\[
\frac{1}{G} \frac{\pa G}{\pa x}(t,x)=\frac{\pa \ln G}{\pa x}(t,x)=\int_0^t B'(Y(s,x))\frac{\pa Y}{\pa x}(s,x) \, ds.
\]

Since the product-rule of calculus is valid for the Stratonovich differentials,
making the change of unknown function $v$ to $g=\exp \{-\Om W_t\}v$ rewrites \eqref{eqstochMcKeanVlonedim}
in terms of $g$ as
\begin{equation}
\label{eqstochMcKeanVlonedimtran}
\dot g_t =\tilde L_t[W](g_t)=\exp \{-\Om W_t\} L_t (\exp \{\Om W_t\} g_t),
\end{equation}
with $\dot v$ denoting the usual derivative of a function $v$ in time $t$. Of course one can obtain the same result using
usual Ito's formula and equation \eqref{eqstochMcKeanVlonedim3a}.
Since the operators $e^{t\Om}$ form a bounded semigroup in $L_1(\R)$,
as well as in $C^k(\R)$ and $C^k_{\infty}(\R)$ whenever $A,B\in C^k(\R)$,
equations \eqref{eqstochMcKeanVlonedimtran} and \eqref{eqstochMcKeanVlonedim}
are equivalent in the strongest possible sense.

To have a more concrete version of \eqref{eqstochMcKeanVlonedimtran} we calculate
\[
\frac{\pa }{\pa x} (\exp \{\Om W_t\} g_t)(x)
=\frac{\pa G}{\pa x}(W_t,x)g_t(Y(W_t,x))+G(W_t,x)\frac{\pa g_t}{\pa Y}(Y(W_t,x))\frac{\pa Y}{\pa x}(W_t,x),
\]
\[
\frac{\pa^2 }{\pa x^2} (\exp \{\Om W_t\} g_t)(x)
=\frac{\pa^2 G}{\pa x^2}(W_t,x)g_t(Y(W_t,x))+2\frac{\pa G}{\pa x}(W_t,x)\frac{\pa g_t}{\pa Y}(Y(W_t,x))\frac{\pa Y}{\pa x}(W_t,x)
\]
\[
+G(W_t,x)\frac{\pa g_t}{\pa Y}(Y(W_t,x))\frac{\pa ^2Y}{\pa x^2}(W_t,x)
+G(W_t,x)\frac{\pa^2 g_t}{\pa Y^2}(Y(W_t,x))\left(\frac{\pa Y}{\pa x}(W_t,x)\right)^2.
\]
Hence,
\begin{equation}
\label{eqstochMcKeanVlonedimtran1}
\tilde L_t[W](g_t)(x)
=\frac12 \tilde \si^2(x)\frac{\pa^2 g_t}{\pa x^2}+ \tilde b(t,x,[g_t]) \frac{\pa g_t}{\pa x} + \tilde c(t,x,[g_t]) g_t,
\end{equation}
with
\begin{equation}
\label{eqstochMcKeanVlonedimtran2}
\tilde \si^2 (x)=\si^2(Y(-W_t,x)) \left(\frac{\pa Y}{\pa z}(W_t,z)|_{z=Y(-W_t,x)}\right)^2,
\end{equation}
\[
\tilde b (t,x,[g])= \left(b (t,z,[\exp \{\Om W_t\} g])\frac{\pa Y}{\pa z}(W_t,z)\right)|_{z=Y(-W_t,x)}
\]
\begin{equation}
\label{eqstochMcKeanVlonedimtran3}
+\left[\frac12\si^2(z) \left(\frac{\pa^2 Y}{\pa z^2}(W_t,z)
+2\frac{\pa \ln G}{\pa z}(W_t,z)\frac{\pa Y}{\pa z}(W_t,z)\right)\right]|_{z=Y(-W_t,x)},
\end{equation}
\[
\tilde c (t,x,[g])= \left.\left(c (t,z,[\exp \{\Om W_t\} g])
+b (t,z,[\exp \{\Om W_t\} g])\frac{\pa \ln G}{\pa z}(W_t,z)\right)\right|_{z=Y(-W_t,x)}
\]
\begin{equation}
\label{eqstochMcKeanVlonedimtran4}
+\left.\left[\frac12\si^2(z) \left(\frac{1}{G}\frac{\pa ^2 G}{\pa z^2}\right)(W_t,z)\right] \right|_{z=Y(-W_t,x)}.
\end{equation}

The formulas above have straightforward extension to $x$ from arbitrary dimension.
The simplification arising from working in one-dimension is as follows:
\[
Y(t,x)=\Phi^{-1} (t+\Phi(x)),
\]
where
\[
\Phi (y)=\int_0^y \frac{dz}{A(z)}.
\]
Hence, under \eqref{eqass1}, \eqref{eqass2} it follows that
\[
\frac{A_1}{A_2}\le \frac{\pa Y}{\pa x}(t,x)\le \frac{A_2}{A_1}, \quad
\left|\frac{\pa ^2Y}{\pa x^2}(t,x)\right| \le \frac{B_1 A_2}{A_1^2}\left(1+ \frac{A_2^2}{A_1^2}\right)
\]
and
\[
\left|\frac{\pa \ln G}{\pa x}(t,x)\right|\le |t| B_1 \frac{A_2}{A_1}
\]
for all $t,x$
and hence
\begin{equation}
\label{eqstochMcKeanVlonedimtran5}
\si_1 \frac{A_1}{A_2} \le \tilde \si (x) \le \si_2 \frac{A_2}{A_1},
\end{equation}
\begin{equation}
\label{eqstochMcKeanVlonedimtran6}
\tilde b (t,x,[g]) \le C(T)(1+\bar W_T),\quad \tilde c (t,x,[g]) \le C(T)(1+\bar W_t),
\end{equation}
with some constants $C(T)$ and $\bar W_T=\max_{t\in [0,T]}|W_t|$.

Thus on any finite interval of time $[0,T]$ equation
\begin{equation}
\label{eqstochMcKeanVlonedimKunTran}
\dot g_t=\tilde L_t [W]g_t
\end{equation}
is the usual nonlinear McKean-Vlasov diffusion equation with uniformly elliptic second order part and
bounded coefficients. For this equation both well-posedness and smooth dependence on initial condition
is known, see e.g. \cite{Ko10} and \cite{KTY14}.
A new point for us is the necessity to have bounds for the expectations of the various relevant objects.
So we shall briefly recall the argument used for the analysis of the sensitivity of equation
\eqref{eqstochMcKeanVlonedimKunTran} paying attention to the latter issue.

Let us first make a precise statement about equation \eqref{eqstochMcKeanVlonedimKunTran}
independently on its link with our initial SPDE.
We shall need the following assumptions:
\begin{equation}
\label{eqassontran1}
\tilde \si_1 \le \tilde \si \le \tilde \si_2, \quad \tilde \si' \le \tilde \si_2,
\end{equation}
\begin{equation}
\label{eqassontran2}
\max\left(\|\tilde b(t,.,[v])\|_{C^1(\R)}, \|\tilde c(t,.,[v])\|_{C(\R)}\right) \le \tilde b(1+\bar W_T)
\end{equation}
and either
\begin{equation}
\label{eqassontran3}
\sup_{\|v\|_{L_1}\le \la}\!\!\max\left(\left\|\frac{\de \tilde b(t,y,[v])}{\de v(.)}\right\|_{C(\R)}\!\!, \left\|\frac{\de \tilde c(t,y,[v])}{\de v(.)}\right\|_{C(\R)}\!\!,
\left\|\frac{\de}{\de v(.)}\frac{\pa \tilde b(t,y,[v])}{\pa y}\right\|_{C(\R)}\right)\!\!\le \tilde C(\bar W_T, \la)
\end{equation}
or
\begin{equation}
\label{eqassontran4}
\sup_{\|v\|_{C(\R)}\le \la}\!\!\max\left(\left\|\frac{\de \tilde b(t,y,[v])}{\de v(.)}\right\|_{L_1(\R)}\!\!,
\left\|\frac{\de \tilde c(t,y,[v])}{\de v(.)}\right\|_{L_1(\R)}\!\!,
\left\|\frac{\de}{\de v(.)}\frac{\pa \tilde b(t,y,[v])}{\pa y}\right\|_{L_1(\R)}\right)\!\!\!\le \tilde C(\bar W_T,\la)
\end{equation}
for some constants $\tilde \si_{1,2}, \tilde b$ and a function $\tilde C$ depending on $\bar W_T=\sup_{t\in [0,T]}|W_t|$ and $\la$.

The idea is to rewrite the Cauchy problem for \eqref{eqstochMcKeanVlonedimKunTran}
with the initial condition $g_0$  in the mild form, that is as a fixed point equation
\begin{equation}
\label{eqstochMcKeanVlonedimKunTranmild1a}
\Phi [g]=g
\end{equation}
for the mapping
\[
\Phi g\mapsto \Phi_t [g](x)
=\int G(t,x,y) g_0 (y) dy
\]
\begin{equation}
\label{eqstochMcKeanVlonedimKunTranmild1}
+\int_0^t G(t-s,x,y) \left[\tilde b (s,y,[g_s])\frac{\pa g_s}{\pa y}(y)+\tilde c(s,y,[g_s])g_s(y)\right] \, ds,
\end{equation}
where $G(t,s,y)$ is the Green function for the Cauchy problem of the equation
\[
\dot g = \frac12 \tilde \si^2(x)\frac{\pa^2 g}{\pa x^2}.
\]
Using integration by parts $\Phi$ rewrites equivalently as
\[
 \Phi_t [g](x)
=\int G(t,x,y) g_0 (y) dy
\]
\begin{equation}
\label{eqstochMcKeanVlonedimKunTranmild2}
+\int_0^t ds \int \left[ G(t-s,x,y) \tilde c(s,y,[g_s])-\frac{\pa}{\pa y}(G(t-s,x,y)\tilde b (s,y,[g_s]))
\right] g_s(y) \, dy,
\end{equation}
from which it is seen that $\Phi$ maps bounded families of functions $g_t$, $t\in [0,T]$, with a given $g_0$, to itself, where
'bounded' can be understood either in sup-norm or in $L_1$-norm.

The Green function $G$ is random, i.e. it depends on $W$. However, by the standard theory
 of the second order equations (see \cite{Itobook} and \cite{PorEi}), assuming \eqref{eqassontran1},
 the function $G$ has the two-sided Gaussian bounds
\[
\frac{C_1}{\sqrt{2\pi t}} \exp \{ -\frac{(x-y)^2}{2tC_3}\}
\le G(t,x,y) \le \frac{C_2}{\sqrt{2\pi t}} \exp \{ -\frac{(x-y)^2}{2tC_4}\},
\]
and the bound for the derivatives
\[
\max\left(\left|\frac{\pa G(t,x,y)}{\pa x}\right|,\left|\frac{\pa G(t,x,y)}{\pa y}\right|\right)
\]
\[
\le C_5 \left|\frac{\pa}{\pa x} \frac{C_2}{\sqrt{2\pi t}} \exp \{ -\frac{(x-y)^2}{2tC_4}\}\right|
\le \frac{C_5 C_2}{C_4 \sqrt t}\frac{|x-y|}{\sqrt t \sqrt{2\pi t}}\exp \{ -\frac{(x-y)^2}{2tC_4}\},
\]
with constants $C_1-C_5$ independent of the noise.

This allows one to infer the estimates
\begin{equation}
\label{eqstochMcKeanVlonedimKunTranmild3}
\|\Phi_t[g^1]-\Phi_t[g_2]\|_{L_1}
\le
\int_0^t ds \frac{1}{\sqrt {t-s}}\|g^1_s-g^2_s]\|_{L_1} \tilde C(\bar W_T,\sup_{s\in [0,T]}\max (\|g^1_s\|_{L_1},\|g^2_s\|_{L_1}))
\end{equation}
in case of conditions \eqref{eqassontran1}, \eqref{eqassontran2}, \eqref{eqassontran3} or
 \begin{equation}
 \begin{split}
\label{eqstochMcKeanVlonedimKunTranmild4}
&\|\Phi_t[g^1]-\Phi_t[g_2]\|_{C(\R)}\\
\le
\int_0^t ds \frac{1}{\sqrt {t-s}}\|g^1_s-g^2_s]\|_{C(\R)} &\tilde C(\bar W_T,\sup_{s\in [0,T]}\max (\|g^1_s\|_{C(\R)},\|g^2_s\|_{C(\R)}))
\end{split}
\end{equation}
in case of conditions \eqref{eqassontran1}, \eqref{eqassontran2}, \eqref{eqassontran4}.

From these estimates one can infer the convergence of the iterates $\Phi^n (g_0)$ in either sup-norm or $L_1$-norm
and hence the existence of the unique solution $g_t \in L_1$ for any initial $g_0 \in L_1$ (and even for any initial
finite measure $g_0$) or of the unique solution $g_t \in C(\R)$ for any initial $g_0 \in C(\R)$ (and even for any initial
bounded measurable $g_0$), whenever one can prove the uniform boundedness of the norms of all iterations.

Let us see how one can get an estimate for the norm of the iterations.
From the definition of $\Phi$ and the estimates of the Green function $G$ we get
\[
\|\Phi_t^n\| \le C\|g_0\| + C\int_0^t (t-s)^{-1/2} \|\Phi_s^{n-1}\| ds,
\]
where $C=C(T)(1+\bar W_T)$ with $C(T)$ a non-random constant, and where
$\|\Phi_t^n\|$ is the norm of the $n$th iteration of $\Phi$ applied initially on $g_0$,
and the norm is either in $C(\R)$ or in $L_1$.
From this we deduce, by a straightforward induction, that
\begin{equation}
\label{eqstochMcKeanVlonedimKunTranmild3}
\|\Phi_t^n\|
\le C\|g_0\| (1+C\sqrt \pi I_{1/2}\1(t) +\cdots +(C\sqrt \pi)^{n-1} I_{(n-1)/2}\1(t)
+C^{n-1} \pi^{n/2} I_{n/2}\1(t)),
\end{equation}
were $I_{n/2}\1(t)$ is the application of the fractional integral of order $n/2$ to the constant function $\1$ (that equals one).
And consequently we get for the limiting norm of the fixed point the bound in terms of a Mittag-Leffler function
and hence eventually in terms of an exponent of $Ct$. Hence, since the expectation of $\exp \{ \bar W_T\}$ is finite,
we can deduce the bound for the expectation of the fixed point yielding the following result.

\begin{theorem}
\label{thwellposMcVlSPDE1}
(i) Under assumptions \eqref{eqassontran1}, \eqref{eqassontran2}, \eqref{eqassontran3} any $T>0$
for any $g_0\in \MC(\R)$ there exists a unique solution
$g_t$ of equation \eqref{eqstochMcKeanVlonedim} on $[0,T]$
such that $g_t\in L_1(\R)$ for all $t>0$, positive whenever $g_0$ is positive, and
\begin{equation}
\label{eq1thwellposMcVlSPDE1}
\|g_t\|_{L_1} \le C_1(T) \exp \{C_1(T) \bar W_T\} \|g_0\|_{\MC(\R)},  \quad \E \|g_t\|_{L_1} \le C_2(T) \|g_0\|_{\MC(\R)}
\end{equation}
with constants $C_{1,2}(T)$.

Moreover,
if $u_0 \in H_1^1$, then
 \begin{equation}
\label{eq2thwellposMcVlSPDE1}
\|g_t\|_{H^1_1} \le C_1(T) \exp \{C_1(T) \bar W_T\} \|g_0\|_{H^1_1},  \quad \E \|g_t\|_{H^1_1} \le C_2(T) \|g_0\|_{H^1_1}.
\end{equation}

Finally, for any $g_0\in \MC(\R)$,
$g_t \in H_1^1$ a.s. for all $t>0$ and, if the bounds on the r.h.s. of \eqref{eqassontran3} do not depend on $\bar W_T$,
 one has the estimate (uniform with respect to the noise)
\begin{equation}
\label{eq3thwellposMcVlSPDE1}
\|g_t\|_{H^1_1} \le C_3(T) \frac{1}{\sqrt t}\|g_0\|_{\MC(\R)}.
\end{equation}

(ii) Under assumptions \eqref{eqassontran1}, \eqref{eqassontran2}, \eqref{eqassontran4},
for any $g_0\in L_{\infty}(\R)$ there exists a unique solution $g_t$ of equation \eqref{eqstochMcKeanVlonedim}
on $[0,T]$ such that $g_t\in C(\R)$ for all $t>0$, and
\begin{equation}
\label{eq4thwellposMcVlSPDE1}
\|g_t\|_{C(\R)} \le C_1(T) \exp \{C_1(T) \bar W_T\} \|g_0\|_{L_{\infty}},  \quad \E \|g_t\|_{C(\R)} \le C_2(T) \|g_0\|_{L_{\infty}}
\end{equation}
with constants $C_{1,2}(T)$.

Moreover,
 if $g_0 \in C^1(\R)$, then
 \begin{equation}
\label{eq5thwellposMcVlSPDE1}
\|g_t\|_{C^1(\R)} \le C_1(T)\exp \{C_1(T) \bar W_T\} \|g_0\|_{C^1(\R)},
 \quad \E \|g_t\|_{C^1(\R)} \le C_2(T) \|g_0\|_{C^1(\R)}.
\end{equation}

Finally, for any $g_0\in L_{\infty}(\R)$,
$g_t \in C^1(\R)$ a.s. for all $t>0$ and, if the bounds on the r.h.s. of \eqref{eqassontran4} do not depend on $\bar W_T$,
 one has the estimate
\begin{equation}
\label{eq6thwellposMcVlSPDE1}
\|g_t\|_{C^1(\R)} \le C_3(T) \frac{1}{\sqrt t}\|g_0\|_{C(\R)}.
\end{equation}

\end{theorem}

\begin{proof}
Let us talk about (i) only, as (ii) is fully analogous.
The proof of the first statement was already sketched above.
The estimates for the norm in $H_1^1$ are obtained from the iterations in a fully analogous way leading
to \eqref{eq2thwellposMcVlSPDE1}. Finally, we get from \eqref{eqstochMcKeanVlonedimKunTranmild1a} the estimate
\[
\|g_t\|_{H_1^1} \le \frac{1}{\sqrt t} \|g_0\|_{L_1} + \int_0^t C(\bar W_T) \frac{1}{\sqrt {t-s}} \|g_s\|_{H_1^1} ds,
\]
so that
\[
\sup_{s\in (0,t]} (\sqrt s \|g_s\|_{H_1^1}) \le \|g_0\|_{L_1}
+ \sqrt t  C(\bar W_T)\int_0^t \frac{1}{\sqrt {t-s}\sqrt s} \sup_{s\in (0,t]} (\sqrt s \|g_s\|_{H_1^1}) \, ds,
\]
and thus
\[
\sup_{s\in (0,t]} (\sqrt s \|g_s\|_{H_1^1}) \le \|g_0\|_{L_1}
+ \sqrt t  C(\bar W_T)\int_0^1 \frac{du}{\sqrt {1-u}\sqrt u} \sup_{s\in (0,t]} (\sqrt s \|g_s\|_{H_1^1}).
\]
If $C(\bar W_T)=C(T)$ actually does not depend on $\bar W_T$ we get for small enough $t$ that
\[
 \sup_{s\in (0,t]} (\sqrt s \|g_s\|_{H_1^1})
 \le \frac{ \|g_0\|_{L_1}}{1- t C(T)}
 \]
 with a constant $C$ implying \eqref{eq3thwellposMcVlSPDE1}.
 And in general we get a similar estimate a.s.

\end{proof}

Our basic objective is to study the sensitivity of the solution $g_t$ with respect to initial data, that is
\begin{equation}
\label{eqvarderdef}
\xi_t(.;x)[g_0]=\frac{\de g_t}{\de g_0(x)}=\frac{d}{dh}|_{h=0} g_t[g_0+h\de_x].
\end{equation}
This can be done in general by analyzing the convergence of the successive approximations to the solutions
\[
\xi^n_t(y;x)[g_0]=\frac{\de \Phi^n_t[g](y)}{\de g_0(x)},
\]
which satisfies the recursion
\[
\xi_t^n(y;x)=G(t,x,y)
\]
\[
+\int_0^t ds \int \left[G(t-s,x,y)\tilde c(s,y,[g_s])-\frac{\pa}{\pa y}(G(t-s,x,y)\tilde b (s,y,[g_s]))\right] \xi_s^{n-1}(y;x) \, dy\\
\]
\[
+\int_0^t ds\iint\frac{\de}{\de g_s(z)}\biggl[G(t-s,x,y) \tilde c(s,y,[g_s])
\]
\begin{equation}
\label{eqstochMcKeanVlonedimKunTransen1}
\left.-\frac{\pa}{\pa y}(G(t{-}s,x,y)\tilde b (s,y,[g_s]))\right] \xi_s^{n-1}(z;x) g_s(y) \, dy \, dz.
\end{equation}
Under the assumptions of Theorem \ref{thwellposMcVlSPDE1}, say (i), we get the recursive estimates for $\xi_n$
in the form
\[
\|\xi_t^n(.;x)\|_{\MC^{sign}(\R)} \le C + C(T, \bar W_t, \sup_{t\in [0,T]} \|g_t\|_{L_1})
\int_0^t (t-s)^{-1/2} \|\xi_s^{n-1}\|_{\MC^{sign}(\R)} ds,
\]
and by linearity the same estimates for the increments $\xi_t^{n+1}(.;x)-\xi_t^{n-1}(.;x)$ in terms of the increments
$\xi_t^n(.;x)-\xi_t^{n-1}(.;x)$ implying the convergence of the sequence $\xi^n$ and hence
the existence of the derivative \eqref{eqvarderdef} almost surely.


To apply Theorem \ref{thwellposMcVlSPDE1} to equation  \eqref{eqstochMcKeanVlonedim},
we have to calculate the variational derivatives of the type
$\de F(\exp \{ \Om W_t\}g)/\de g$ in terms of the derivatives of $F$. To this end, let us first find out, how
the transformation $e^{t\Om}$ acts on measures (rather than functions). For any functions $v\in L_1$, $\phi \in C(\R)$ we have
\[
(\phi, v)= \int \phi (x) e^{\Om t} v(x) \, dx = \int G(t,x) \phi (x) v(Y(t,x)) dx
\]
\[
=\int G(t,Y(-t,z)) \phi (Y(-t,z)) v(z) \frac{\pa Y(-t,z)}{\pa z} \, dz,
\]
from which the extension to measures $v$ is directly seen. Thus, for any measure $g$, we get
\[
(\phi, \frac{\de }{\de g(x)} e^{t\Om} g)=(\phi, e^{t\Om} \de_x)
=G(t,Y(-t,x)) \phi (Y(-t,x)) \frac{\pa Y(-t,x)}{\pa x},
\]
so that
\[
\frac{\de }{\de g(x)} e^{t\Om} g=G(t,Y(-t,x)) \frac{\pa Y(-t,x)}{\pa x} \de_{Y(-t,x)}.
\]
Consequently,
\[
\frac{\de}{\de g(x)} F ( \exp \{ \Om W_t\}g)=\int \left.\frac{\de F(\mu)}{\de \mu (z)}\right|_{\mu=\exp \{ \Om W_t\}g}
(e^{\Om W_t} \de_x)(z) \, dz
\]
\begin{equation}
\label{eqvarderKuTran}
= \left.\frac{\de F(\mu)}{\de \mu (z)}\right|_{\mu=\exp \{ \Om W_t\}g}^{z=Y(-W_t,x)} G(W_t,Y(-W_t,x)) \frac{\pa Y(-W_t,x)}{\pa x}.
\end{equation}

Using this formula, equation \eqref{eqstochMcKeanVlonedimtran2} - \eqref{eqstochMcKeanVlonedimtran4},
and the convergence of sequence \eqref{eqstochMcKeanVlonedimKunTransen1},
we obtain the following result as a consequence of Theorem \ref{thwellposMcVlSPDE1}.

\begin{theorem}
\label{thwellposMcVlSPDE2}
Assume \eqref{eqass1} -- \eqref{eqass4} hold and a $T>0$ given. Then

(i) For any $v_0\in \MC^{sign}(\R)$ there exists a unique solution
$v_t$ of equation \eqref{eqstochMcKeanVlonedim} on $[0,T]$
such that $v_t\in L_1(\R)$ for all $t>0$, positive whenever $v_0$ is positive, and
\begin{equation}
\label{eq1thwellposMcVlSPDE2}
\|v_t\|_{L_1} \le C_1(T) \exp \{C_1(T) \bar W_T\} \|v_0\|_{\MC(\R)},  \quad \E \|v_t\|_{L_1} \le C_2(T) \|v_0\|_{\MC(\R)}
\end{equation}
with constants $C_{1,2}(T)$.

(ii) If $v_0 \in H_1^1$, then
 \begin{equation}
\label{eq2thwellposMcVlSPDE2}
\|v_t\|_{H^1_1} \le C_1(T) \exp \{C_1(T) \bar W_T\} \|v_0\|_{H^1_1},  \quad \E \|g_t\|_{H^1_1} \le C_2(T) \|g_0\|_{H^1_1}.
\end{equation}

(iii) For any $v_0\in \MC(\R)$,
$v_t \in H_1^1$ a.s. for all $t>0$ and, if the bounds on the r.h.s. of \eqref{eqassontran3} do not depend on $\bar W_T$,
 one has the estimate
\begin{equation}
\label{eq3thwellposMcVlSPDE2}
\|v_t\|_{H^1_1} \le C_3(T) \frac{1}{\sqrt t}\|v_0\|_{\MC(\R)}.
\end{equation}

(iv) The variational derivative $\xi_t(.;x)[v_0]=\frac{\de v_t}{\de v_0(x)}$
of the solution $v_t$ with respect to initial data exists a.s. as a measure of finite total variation.
\end{theorem}

\section{Sensitivity for McKean-Vlasov SPDEs}
\label{secsen}

We shall discuss in more detail the sensitivity of McKean-Vlasov SPDE
\eqref{eqstochMcKeanVlonedim} reducing our attention to a more specific case of $L$ having the form of
a dual second order operator, namely to the equation

\begin{equation}
\label{eqstochMcKeanVlonedimdu}
dv = L'_{t,v}v \, dt -\nabla (A(x) v) \circ dW_t,
\end{equation}
where
\begin{equation}
\label{eqstochMcKeanVlonedimdu2}
L_{t,v}\phi =\frac12 \si^2(x)\frac{\pa^2 \phi}{\pa x^2}+ b(t,x,[v]) \frac{\pa \phi}{\pa x},
\end{equation}
and $L'_{t,v}$, its dual, defined as
\begin{equation}
\label{eqstochMcKeanVlonedimdu2}
L'_{t,v}u =\frac12 \frac{\pa^2 }{\pa x^2} (\si^2(x)u(x))-\frac{\pa }{\pa x} (b(t,x,[v])u(x)).
\end{equation}

Equation \eqref{eqstochMcKeanVlonedimdu} is a particular case of \eqref{eqstochMcKeanVlonedim}, so that
the theory of the previous section applies. Moreover, this equation naturally rewrites in the weak
form as
\begin{equation}
\label{eqstochMcKeanVlonedimduweak}
d(\phi,v) = (L_{t,v}\phi,v) \, dt + (\Om \phi, v) \circ dW_t,
\end{equation}
with $\Om \phi = A(x) \nabla \phi$.

Making in \eqref{eqstochMcKeanVlonedimdu} the change of function to $g=\exp \{-\Om' W_t\} v$,
where $\Om'=-\nabla \circ A(x)$ is the dual to $\Om$, leads to the equation
\begin{equation}
\label{eqstochMcKeanVlonedimdutran}
\dot g = \exp \{-\Om' W_t\} L'_{t,\exp \{\Om' W_t\}g} \exp \{\Om' W_t\}g \, dt,
\end{equation}
or in the weak form
\begin{equation}
\label{eqstochMcKeanVlonedimdutranweak}
\frac{d}{dt} (\phi,g) = (\tilde L_{t,\exp \{\Om' W_t\}g} \phi,g)=(\exp \{\Om W_t\} L_{t,\exp \{\Om' W_t\}g} \exp \{-\Om W_t\}\phi, g).
\end{equation}

Notice now that the operator $\Om =A(x) \nabla$ coincides with  \eqref{eqstochMcKeanVlonedim1} with
vanishing $B$, and hence the corresponding transformation $e^{t\Om}$ given by \eqref{eqmethcharac} has $G=1$
and hence the estimates \eqref{eqstochMcKeanVlonedimtran6} does not contain $\bar W_T$.
Moreover,
\begin{equation}
\label{eqdualchange}
e^{t\Om'}v(z)=v(Y(-t,z))\frac{\pa Y(-t,z)}{\pa z},
\end{equation}
so that the estimate \eqref{eqstochMcKeanVlonedimtran6} for the operator $\Om'$ also does not contain $\bar W_T$.
Consequently Theorem \ref{thwellposMcVlSPDE2} for equation \eqref{eqstochMcKeanVlonedimdu}
holds in its strongest form
containing estimate \eqref{eq3thwellposMcVlSPDE2}.
Moreover, general formulas \eqref{eqstochMcKeanVlonedimtran1}-\eqref{eqstochMcKeanVlonedimtran4}
simplify essentially for $B=0,c=0$ allowing us to rewrite $\tilde L$ from \eqref{eqstochMcKeanVlonedimdutranweak} as
\begin{equation}
\label{eqstochMcKeanVlonedimdutranweak1}
\tilde L_{t,\exp \{\Om' W_t\}g} \phi
=\frac12 \si^2(Y(W_t,x))\frac{\pa^2 \phi}{\pa x^2}+ \tilde b(t,x,[g]) \frac{\pa \phi}{\pa x},
\end{equation}
where $Y(t,x)$ solves the ODE $\dot Y=-A(Y)$ with the initial condition $Y(0,x)=x$ and
\begin{equation}
\label{eqstochMcKeanVlonedimdutranweak2}
\tilde b(t,x,[g]) =\left. \left(b(t,z, [\exp \{\Om' W_t\}g])\frac{\pa Y}{\pa z}(-W_t,z)
+\frac12 \si^2(z)\frac{\pa^2 Y}{\pa z^2}(-W_t,z) \right)\right|_{z=Y(W_t,x)}.
\end{equation}

Furthermore, as operator \eqref{eqstochMcKeanVlonedimdutranweak1} is the generator of a diffusion,
its solution cannot increase the sup-norm, and hence the solution to equation \eqref{eqstochMcKeanVlonedimdu}
does not increase the $L_1$-norm (or, equivalently, the $\MC(\R)$-norm).

To study sensitivity of equation \eqref{eqstochMcKeanVlonedimdu}, we can now apply the results of \cite{Ko10}
and \cite{KTY14} to equation \eqref{eqstochMcKeanVlonedimdutranweak}. However, these results
 yield the existence of the derivatives with respect to initial data
for almost all $W_t$, and we are interested here in the expectation of all bounds. Therefore,
we sketch briefly the approach of \cite{KTY14} to see how the estimates for the expectation arise.

Let us differentiate \eqref{eqstochMcKeanVlonedimdutranweak1} to get the equation for the derivatives
\[
\xi_t(.;x)[g_0]=\frac{\de g_t}{\de g_0(x)}=\frac{d}{dh}|_{h=0} g_t[g_0+h\de_x].
\]
the existence of these derivatives is already proved in Theorem \ref{thwellposMcVlSPDE2}.

Using \eqref{eqvarderKuTran} we get
\[
(\phi, \dot \xi_t(.;x)[g_0])=(\tilde L_{t,\exp \{\Om' W_t\}g} \phi,\xi_t(.;x)[g_0])
\]
\begin{equation}
\label{eqdualvarder0}
+\iint\!\!\left.\left(\frac{\de b(t,z, [\exp\{ \Om'W_t\}g])}{\de g (r)}\frac{\pa Y(-W_t,z)}{\pa z}\right)\right|_{z=Y(W_t,y)}\!\!
\xi_t(r;x)[g_0] \frac{\pa \phi}{\pa y} g_t(y) dy dr.
\end{equation}

Thus the evolution of $\xi_t$, considered as measures, is dual to the evolution on functions defined
in the inverse time via the equation
\[
\dot \phi_s = -\tilde L_{s,\exp \{\Om' W_t\}g} \phi_s
\]
\begin{equation}
\label{eqdualvarder1}
- \int\!\!\left.\left(\frac{\de b(t,z, [\exp\{ \Om'W_t\}g])}{\de g (.)}\frac{\pa Y(-W_t,z)}{\pa z}\right)\right|_{z=Y(W_t,y)}
\!\!\frac{\pa \phi_s}{\pa y} g_s(y) \, dy.
\end{equation}
This equation defines the backward propagator $U^{s,t}$, $s\le t$, on $C^1(\R)$, such that $U^{s,t}\phi$ is the solution
to equation  \eqref{eqdualvarder1} with the terminal condition $\phi$ at time $t$, and $U^{s,t}=(V^{t,s})'$, where
$V^{t,s}$ is the forward propagator yielding the solution to equation \eqref{eqdualvarder0}. To see that $U^{s,t}$
is well defined as claimed, let us write \eqref{eqdualvarder1} more explicitly. Namely, as follows from
\eqref{eqdualchange} and \eqref{eqvarderKuTran},
\begin{equation}
\label{eqvarderKuTran}
\frac{\de}{\de g(x)} F ( \exp \{ \Om' W_t\}g)
= \left.\frac{\de F(\mu)}{\de \mu (z)}\right|_{\mu=\exp \{ \Om' W_t\}g, z=Y(W_t,x)}.
\end{equation}
Consequently, \eqref{eqdualvarder1} rewrites as
\[
\dot \phi_s(p) = -\tilde L_{s,\exp \{\Om' W_t\}g} \phi_s(p)
\]
\begin{equation}
\label{eqdualvarder2}
-\int \left.\frac{\de b(t,z, \mu)}{\de \mu(r)}\right|_{\mu=\exp \{ \Om' W_t\}g_s}^{r=Y(W_t,p)}
\left.\frac{\pa Y(-W_t,z)}{\pa z}\right|_{z=Y(W_t,p)}
\frac{\pa \phi_s}{\pa y} g_s(y) \, dy.
\end{equation}

From the form of $\tilde L$ it is seen that it generates a Feller semigroup in $C(\R)$ with an invariant domain $C^1(\R)$,
and the second term  of \eqref{eqdualvarder2} is a bounded operator in $C^1(\R)$, due to the assumption
\eqref{eqass4} on the norm of $\de b(t,y,[v])/\de v(.)$ in $C^1(\R)$. Thus one can solve
\eqref{eqdualvarder2} by the standard perturbation theory showing that
\[
\|U_{s,t}\phi\|_{C^1(\R)} \le C(T) \|\phi\|_{C^1(\R)}
\]
with $C(T)$ depending only on the $\|g_0\|_{\MC(\R)}$ and not on the noise implying that
\[
\E \|U_{s,t}\phi\|_{C^1(\R)} \le C (T) \|\phi\|_{C^1(\R)}.
\]
Consequently the dual propagator $V^{t,s}$ defining the solution to $\xi$ in equation \eqref{eqdualvarder0}
is a bounded propagator in the dual space $(C^1(\R))'$, both a.s. and on average. Similarly,
assuming additional smoothness of coefficients we can claim that $U^{s,t}$ acts in $C^2(\R)$.
Finally, by moving the derivative of $\phi$ in the second term of \eqref{eqdualvarder2} to $g$ via the integration by parts
and using \eqref{eq3thwellposMcVlSPDE2} allows one to show that $U^{s,t}$ acts as a bounded semigroup in $C(\R)$.
Therefore, the estimates
\begin{equation}
\label{eq1thwellposMcVlSPDE3}
\|U_{s,t}\phi\|_{C^k(\R)} \le C(T) \|\phi\|_{C^k(\R)},
\quad \E \|U_{s,t}\phi\|_{C^k(\R)} \le C (T) \|\phi\|_{C^k(\R)}
\end{equation}
hold for $k=0,1,2$ with constants $C(T)$ depending only on the $\|v_0\|_{\MC(\R)}$.
The dual propagator $V^{t,s}$, solving equation \eqref{eqdualvarder0},
 is a bounded propagator in the dual spaces
$(C(\R))'$, $(C^1(\R))'$ and $(C^2(\R))'$ with bounds independent of the noise.
This implies that the dual propagator $V^{t,s}$, solving equation \eqref{eqdualvarder0},
 is a bounded propagator in the dual spaces
$(C(\R))'$, $(C^1(\R))'$ and $(C^2(\R))'$ with bounds independent of the noise.
Taking finally into account that
\[
\xi_0(.,x)=\de_x \in \MC(\R),
\quad \frac{\pa}{\pa x} \xi_0(.,x)=\de'_x \in (C^1(\R))',
\quad \frac{\pa^2}{\pa x^2} \xi_0(.,x) \in (C^2(\R))',
\]
leads us to the following result.

\begin{theorem}
\label{thwellposMcVlSPDE3}
Let $T>0$ and
\begin{equation}
\label{eqass1a}
0<\si_1 \le \si(x) \le \si_2, \quad 0<A_1 \le A(x) \le A_2,
\end{equation}
that $\si \in C^2(\R), A\in C^3(\R)$ so that
\begin{equation}
\label{eqass2a}
\sup_x \max(|\si'(x)|, |\si''(x)|, |A'(x)|, | A''(x)|,|A'''(x)|)\le B_1.
\end{equation}
Let
\begin{equation}
\label{eqass3a}
 \| b(t,.,[v])\|_{C^2(\R)} \le  b_1,
\end{equation}
\begin{equation}
\label{eqass4a}
\sup_{t,y}\sup_{\|v\|_{\MC(\R)}\le \la}
\max\left(\left\|\frac{\de b(t,y,[v])}{\de v(.)}\right\|_{C^2(\R)},
\left\|\frac{\de}{\de v(.)}\frac{\pa  b(t,y,[v])}{\pa y}\right\|_{C(\R)}\right) \le C(\la),
\end{equation}
with some constants $\si_1, \si_2, A_1, A_2, B_1, b_1$ and a function $C(\la)$. Then the following holds:

(i) For any $v_0\in \MC^{sign}(\R)$ there exists a unique solution
$v_t$ of equation \eqref{eqstochMcKeanVlonedim} on $[0,T]$
such that $v_t\in L_1(\R)$ for all $t>0$, positive whenever $v_0$ is positive, and
with the norm not exceeding $\|v_0\|_{\MC(\R)}$ for all realization of the noise $W$.
Moreover,
 $v_t \in H_1^1$ for all $t>0$ and the following estimates hold
 \begin{equation}
\label{eq2thwellposMcVlSPDE2}
\|v_t\|_{H^1_1} \le C(T) \|v_0\|_{H^1_1},  \quad \E \|v_t\|_{H^1_1} \le C(T) \|v_0\|_{H^1_1},
\end{equation}
\begin{equation}
\label{eq3thwellposMcVlSPDE3}
\|v_t\|_{H^1_1} \le C(T) \frac{1}{\sqrt t}\|v_0\|_{\MC(\R)}, \quad \E \|v_t\|_{H^1_1} \le C(T) \frac{1}{\sqrt t}\|v_0\|_{\MC(\R)}.
\end{equation}

(ii) The variational derivative $\xi_t(.;x)[v_0]=\frac{\de v_t}{\de v_0(x)}$
of the solution $v_t$ with respect to initial data
are well defined as elements of $L_1(\R)$ for any $x$ and $t>0$,
 and their first and second derivatives with respect to $x$ are bounded elements
of the dual spaces $(C^1(\R))'$ and $(C^2(\R))'$ respectively, so that
\begin{equation}
\label{eq1thwellposMcVlSPDE4}
\|\xi_0(.,x)\|_{L_1} \le C(T),
\quad \|\frac{\pa}{\pa x} \xi_0(.,x)\|_{(C^1(\R))'} \le C(T),
\quad \|\frac{\pa^2}{\pa x^2} \xi_0(.,x)\|_{(C^2(\R))'} \le C(T)
\end{equation}
with constants $C(T)$ depending only on the norm $\|v_0\|_{\MC(\R)}$ and independent of the noise.

\end{theorem}

We are also interested in the second derivatives of the solutions $v_t$ with respect to initial data:
\begin{equation}
\label{eqdefdernachusl2}
\eta_t(.;x_1,x_2)= \frac{d}{dh}|_{h=0}\xi_t(.;x_1)[v_0+h\de_x]
=\frac{\pa^2}{\pa h_1 \pa h_2}|_{h_1=h_2=0}\mu_t[v_0+h_1\de_{x_1}+h_2\de_{x_2}].
\end{equation}

For this $\eta$ we get the following equation differentiating \eqref{eqdualvarder0}
with respect $g_0$;

\[
(\phi, \dot \eta_t(.;x_1,x_2))=(\tilde L_{t,\exp \{\Om' W_t\}g} \phi,\eta_t(.;x_1,x_2))
\]
\[
+\iint\!\!\left.\left(\frac{\de b(t,z, [\exp\{ \Om'W_t\}g])}{\de g (r)}\frac{\pa Y(-W_t,z)}{\pa z}\right)\!\right|_{z=Y(W_t,x)}
\eta_t(r;x_1,x_2) \frac{\pa \phi}{\pa y} g_t(y) dy dr
\]
\begin{equation}
\label{eqdualvarder2}
+\iint\!\!\!\left.\left(\frac{\de b(t,z, [\exp\{ \Om'W_t\}g])}{\de g (r)}\frac{\pa Y(-W_t,z)}{\pa z}\right)\!\right|_{z=Y(W_t,x)}
\frac{\pa \phi}{\pa y} \tilde{\xi}_t(r,y;x_1,x_2) dy dr
\end{equation}
\[
+\iint\!\!\!\left.\left(\frac{\de^2 b(t,z, [\exp\{ \Om'W_t\}g])}{\de g (r_1) \de g(r_2)}\frac{\pa Y(-W_t,z)}{\pa z}\right)\!\right|_{z=Y(W_t,x)}
\!\!\frac{\pa \phi}{\pa y} \xi_t(r_2;x_2)\xi_t(r_1;x_1) g_t(y) \, dy dr_1 dr_2,
\]
where $\tilde{\xi}_t(r,y;x_1,x_2)=[\xi_t(r;x_2)\xi_t(y;x_1)+\xi_t(y;x_2)\xi_t(r;x_1)]$.

The well-posedness of this equations and then the existence of the derivative \eqref{eqdefdernachusl2}
follows as above. However, we need also the existence and bounds for the derivatives of $\eta$ with respect to $x_1,x_2$.

\begin{theorem}
\label{thsenseSPDE2}
Under assumption of Theorem \ref{thwellposMcVlSPDE3} let
$b \in C^{2,1\times 1}(\MC_1^{sign}(\R))$ as a function of $v$ with all bounds uniform in other variables.
 Then for any $v_0\in \MC_1(\R)$
the derivative \eqref{eqdefdernachusl2} is well defined for all $t>0,x$ and the following bounds hold

\begin{equation}
\label{eq1thsenseSPDE2}
\|\frac{\pa^{\al} }{\pa x_1^{\al}}\frac{\pa^{\be} }{\pa x_2^{\be}}\eta_t(.;x_1,x_2)\|_{[C^2(\R)]'} \le C(T)
\end{equation}
with $\al, \be \le 1$ and
with some (random) constants $C$ depending on the time horizon $T$, but not on the noise $W$.
\end{theorem}

\begin{proof}
In the light of the properties of $\xi$
from Theorem \ref{thwellposMcVlSPDE3} it is straightforward to see, differentiating equation \eqref{eqdualvarder2} with respect to $x_1$ and $x_2$
that the assumptions made on $b$ is precisely the one needed to make all terms not containing $\eta$ uniformly bounded
(due to the product structure of $\xi$ entering the equation for $\eta$),
so that they can be written by the usual perturbation arguments.
\end{proof}

\section{On the domain of the Markov semigroups generated by the McKean-Vlasov SPDEs}
\label{secMar}

Since equation \eqref{eqstochMcKeanVlonedimdu} its solutions defines a Markov process, in fact a measure-valued diffusion,
the corresponding Markov propagator being given on the continuous functionals of measures in the usual way:
\begin{equation}
\label{eq1thsenseSPDEMarkovsem}
U^{s,t}F(v)= \E F(v_t(v,[W])),
\end{equation}
where $v_t$ is the solution to \eqref{eqstochMcKeanVlonedimdu} for $t>s$ with given $v=v_s$ at time $s$.

We use the same letter $U$ that was used for the propagators discussed in the proof if Theorem \ref{thwellposMcVlSPDE3}
which should not cause any confusion, as $U$ is used in the sense of \eqref{eq1thsenseSPDEMarkovsem} everywhere, except
in the intermediate discussion leading to  Theorem \ref{thwellposMcVlSPDE3}.

The main conclusion we need from the sensitivity analysis developed above is the invariance of the set of
smooth functionals under this propagator, that is the following fact:

\begin{theorem}
\label{thsenseSPDEMarkovsem}
Under assumption of Theorem \ref{thsenseSPDE2}
the spaces of functionals $C^{1,2} (\MC_{\la}^{sign} (\R))$ and its intersection
 with $C^{2, 1\times 1} (\MC_{\la}^{sign} (\R))$
are invariant under the action of the operators \eqref{eq1thsenseSPDEMarkovsem}, so that
\begin{equation}
\label{eq2thsenseSPDEMarkovsem}
\|U^{s,t} F\|_{C^{1,2}(\MC_{\la}^{sign}(\R))} \le C(T) \|F\|_{C^{1,2}(\MC_{\la}^{sign}(\R))},
\end{equation}
\begin{equation}
\label{eq3thsenseSPDEMarkovsem}
\|U^{s,t} F\|_{C^{2,1\times 1}(\MC_{\la}^{sign}(\R))}
\le C(T) \left(\|F\|_{C^{2,1\times 1}(\MC_{\la}^{sign}(\R))}+\|F\|_{C^{1,2}(\MC_{\la}^{sign}(\R))}\right)
\end{equation}
with a constant $C(T)$.
\end{theorem}

\begin{proof}
It follows from Theorem \ref{thwellposMcVlSPDE3} and the formula
\[
\frac{(\de U^{s,t}(F))(v) }{\de v(x)}=\E \frac{\de F (v_t)}{\de v(x)}=\E \int_{\R} \frac{\de F (v_t)}{\de v_t(z)}\xi_t(z;x)[v] \, dz,
\]
that
\[
\|U^{s,t} F\|_{C^{1,2}(\MC_{\la}^{sign}(\R))}
\le \E \left\| \int_{\R} \frac{\de F (v_t)}{\de v_t(z)}\xi_t(z;.)[v] \, dz \right\|_{C^2(\R)}
\]
\[
\le \E  \left\| \frac{\de F (v_t)}{\de v_t(.)}\right\|_{C^2(\R)}
 \left\| \frac{\pa ^2}{\pa x^2}\xi_t(.;x)[v] \right\|_{(C^2(\R))'} \le C(T).
\]
It follows from Theorem \ref{thsenseSPDE2} and the formula
\[
\frac{(\de^2 U^{s,t}(F))(v) }{\de v(x)\de v(y)}= \E \frac{\de^2 F (v_t)}{\de v(x)\de v (y)}
\]
\[
=\E \int _{\R}\frac{\de F (v_t)}{\de v_t(z)}\eta_t(z;x,y)[v_0] dz
+\E \int_{\R^2} \frac{\de^2 F (v_t)}{\de v_t(z) \de v_t(w)}\xi_t(z;x)[v_0]\xi_t(w;y)[v_0] \, dz dw,
\]
that
\[
\|U^{s,t} F\|_{C^{2,1\times 1}(\MC_{\la}^{sign}(\R))}
\le \E \left\| \frac{\de^2 F (v_t)}{\de v(.)\de v (.)} \right\|_{C^{1\times 1}(\R^2)}
\]
\[
\le \left\|\frac{\pa^{\al} }{\pa x_1^{\al}}\frac{\pa^{\be} }{\pa x_2^{\be}}\eta_t(.;x_1,x_2)\right\|_{[C^2(\R)]'}
\|F\|_{C^{1,2}(\MC_{\la}^{sign}(\R))}
\]
\[
+\left\| \frac{\pa ^{\al}}{\pa x^{\al}}\xi_t(.;x)[v] \right\|_{(C^1(\R))'}
\left\| \frac{\pa ^{\be}}{\pa y^{\be}}\xi_t(.;y)[v] \right\|_{(C^1(\R))'}
\left\| \frac{\de^2 F (v_t)}{\de v_t(.) \de v_t(.)}\right\|_{C^{2,1\times 1}(\MC_{\la}^{sign}(\R))},
\]
leading to \eqref{eq3thsenseSPDEMarkovsem}.
\end{proof}

\section{Proof of Theorem \ref{th2}}

Let us return to our initial equation \eqref{eqstartSDEN}.
By the standard assumption of the Lipschitz continuity of all coefficients, equation
\eqref{eqstartSDEN} is well-posed in $\R^N$ and specifies a Feller diffusion
and the corresponding backward and forward propagators $U_N$, $V_N$ given by \eqref{eqbackpropNpart}, \eqref{eqforpropNpart}.
We are interested in the limit of this diffusion as $N\to \infty$.

Applying Ito's formula we obtain the generator of the diffusion specified by \eqref{eqstartSDEN}:
\begin{equation}
\label{eqstartSDENgen}
A_Nf(x_1, \cdots , x_N)=\sum_{j=1}^N B^i_{\mu} f +\sum_{i<j} \si_{com}(x_i)\si_{com}(x_j) \frac{\pa ^2 f}{\pa x_i \pa x_j},
\end{equation}
where $\mu=(\de_{x_1}+\cdots + \de_{x_N})/N$ and
\begin{equation}
\label{eqstartSDENgen1}
B_{\mu} g(x)=b(t,x, \mu, u(t,x,\mu))\frac{\pa g}{\pa x}
+\frac12(\si^2_{com}(x)+\si^2_{ind}(x)) \frac{\pa ^2 g}{\pa x^2},
\end{equation}
with $B_{\mu}^i$ denoting the action of $B_{\mu}$ on the $i$th coordinate of $f$.

Here and everywhere by a time-dependent generator, say $A_N$ above, of a non-homogeneous Markov process
we mean a time-dependent family  of operators such that for $f$ from some invariant dense subspace
of bounded continuous functions the equation
\[
\frac{d}{ds} U_N^{s,t} f =-A_N U_N^{s,t} f
\]
holds for $s\le t$. In the case of the $N$-particle diffusion, the invariant subspace can be usually taken
to be the space of twice differentiable functions
(which are invariant if $\si$ and $b$ are twice and once differentiable respectively). In the case of
the limiting measure-valued process the invariant domains will be given by the subspaces $C^{2, k\times k}(\MC^{sign})$.

The first term in \eqref{eqstartSDENgen} can be considered as describing a diffusion arising
from the system of particles with a mean-field interaction and the second term as giving an additional binary interaction
(though not of a standard potential type that can be easily included in the mean-field interaction).

By the standard inclusion \eqref{eqstandinclusionpartmeasure}, the process specified by
\eqref{eqstartSDENgen} can be equivalently considered as a measure-valued process defined on
the set of linear combinations $\PC_N(\R)$ of the Dirac atomic measures. On the level of propagators this correspondence
arises from the identification of symmetric functions $f$ on $X^N$ with the functionals $F=F_f$ on $\PC_N(\R)$
via the equation
\[
f(x_1, \cdots , x_N)= F_f[(\de_{x_1}+\cdots + \de_{x_N})/N].
\]
To recalculate the generator \eqref{eqstartSDENgen} in terms of functionals $F$ on measures we use
the following simple formulas for differentiation of functionals on measures (proofs can be found e.g. in \cite{Ko10}):
for $\mu= h(\de_{x_1}+\cdots + \de_{x_N})$ with $h=1/N$
\begin{equation}
\label{eqdifffunctiononmeas1}
\frac{\pa}{\pa x_j} F(\mu)=h \frac{\pa}{\pa x_j} \frac{\de F(\mu)}{\de \mu(x_j)},
\end{equation}
\begin{equation}
\label{eqdifffunctiononmeas2}
\frac{\pa ^2}{\pa x_j^2} F(\mu)=h \frac{\pa^2}{\pa x_j^2} \frac{\de F(\mu)}{\de \mu(x_j)}
+ h^2 \frac{\pa ^2}{\pa y \pa z} \left.\frac{\de^2 F(\mu)}{\de \mu(y) \de \mu (z)}\right|_{y=z=x_j},
\end{equation}
\begin{equation}
\label{eqdifffunctiononmeas3}
\frac{\pa ^2}{\pa x_i \pa x_j} F(\mu)=h^2 \frac{\pa ^2}{\pa x_i \pa x_j} \frac{\de^2 F(\mu)}{\de \mu(x_i) \de \mu (x_j)},
\quad i\neq j.
\end{equation}

Applying these formulas in conjunction with the obvious identity
\begin{equation}
\label{eqcomblemma1}
h^2 \sum_{i<j: i,j\in \{1, \ldots, N\}} \phi(x_i,x_j) =
\frac{1}{2} \int \int \phi(z_1, z_2) \mu(dz_1)\mu(dz_2)
- \frac{h}{2} \int \phi(z,z) \mu(dz),
\end{equation}
leads to the following expression of $A_N$ in terms of $F(\mu)$ (for details and more general calculations see \cite{Ko10}):
\begin{equation}
\label{eqstartSDENgenonmes}
A_NF(\mu)=\La_{lim} F(\mu) +\frac{1}{N} \La_{corr}F(\mu),
\end{equation}
with
\begin{equation}
\begin{split}
\label{eqstartSDENgenonmes1}
\La_{lim} F(\mu) &=\int_{\R}\!\!\left( B_{\mu} \frac{\de F}{\de \mu(.)}\right)(y)\mu(dy)\\
&+\frac12 \int_{\R^2}\si_{com}(y)\si_{com}(z) \frac{\pa^2}{\pa y \pa z}\frac{\de ^2F}{\de \mu(y) \de \mu(z)}\mu(dy)\mu(dz),
\end{split}
\end{equation}
\begin{equation}
\label{eqstartSDENgenonmes2}
\La_{corr} F(\mu) =\frac12 \int_{\R} \si^2_{ind}(x)
\left.\frac{\pa^2}{\pa y \pa z}\frac{\de^2 F(\mu)}{\de \mu(y) \de \mu (z)}\right|_{y=z=x} \mu (dx).
\end{equation}

Thus we have an explicit expression for the limit of $A_N$ as $N\to \infty$ and for the correction term,
which are well defined for functional $F$ from the spaces $C^{1,2}(\MC^{sign})\cap C^{2,1\times 1}(\MC^{sign})$.

It is straightforward to check by Ito's formula that the operator $\La_{lim}$ generates the measure-valued
process defined by the solution of equation \eqref{eqstartSDEN}. Hence we have the convergence of the generators
of $N$-particle approximations to the generator of the process given by \eqref{eqstartSDEN} on
the space $C^{1,2}(\MC^{sign})\cap C^{2,1\times 1}(\MC^{sign})$ with the uniform rate of convergence of order $1/N$.

But according to Theorem \ref{thsenseSPDEMarkovsem}, the propagator of the process generated by
\eqref{eqstartSDEN} acts by bounded operators on this subspace.
Hence Theorem \ref{th2} follows from the standard representation of the difference of two propagators
in terms of the difference of their generators:
\begin{equation}
\label{PropergatorProperty}
U_N^{t,r}-U^{t,r}=\int_t^rU_N^{t,s}(A_N-\La_{lim})_sU^{s,r}ds.
\end{equation}

\section{Proof of Theorem \ref{th3}}
\label{secth1}

The well-posedness of the process on pairs $(x,\mu)$ solving equations
\eqref{eqstartSDENtag} and \eqref{eqlimMacVlaSPDEtag} is straightforward once the well-poesdness of the process
solving   \eqref{eqlimMacVlaSPDEtag} is proved, because equation \eqref{eqlimMacVlaSPDEtag}
does not depend on $x$, and once it is solved, equation
\eqref{eqstartSDENtag} is just a usual Ito's equation. Straightforward extension of the above calculations
for the generator of the process solving \eqref{eqlimMacVlaSPDEtag} show that the process solving
\eqref{eqstartSDENtag} - \eqref{eqlimMacVlaSPDEtag} is generated by the operator
\[
\La_{lim}F(x,\mu)+\tilde \La_{Lim}F(x,\mu),
\]
where $\La_{lim}$ is given by \eqref{eqstartSDENgenonmes1} and acts on the variable $\mu$,
\[
\tilde \La_{lim} F(x,\mu) =b(t,x, \mu, u_t^{ind}(x,\mu))\frac{\pa F}{\pa x}
+\frac12(\si_{ind}^2+\si_{com}^2)(x)\frac{\pa ^2F}{\pa x^2}
\]
\begin{equation}
\label{eqstartSDENgenonmes1tag}
+\int \si_{com}(x)\si_{com}(y) \frac{\pa^2}{\pa x \pa y}\frac{\de F}{\de \mu(y)}\mu(dy),
\end{equation}
and with the same correction term \eqref{eqstartSDENgenonmes2}.
Thus  the proof of Theorem \ref{th3} is the same as for Theorem \ref{th2}.

\section{Proof of Theorem \ref{th1}}

Let $u_1$ be any adaptive control of the first player and $V_1$ the corresponding payoff
 in the game of $N$ players, where all other players are using
$u_{com}(t,x,\mu)$ arising from a solution to \eqref{eqMFGconsistcomnoise1},\eqref{eqsforbackboundary}.
Then $V_1\ge V_2$, where $V_2$ is obtained by playing optimally, that is using control $u_2$ arising from
the solution to \eqref{eqHJBcomnoise1}. By Theorem \ref{th3},
\[
|V_2-V_{2,lim}| \le C/N,
\]
where $V_{2,lim}$ is obtained by playing $u_2$ in the limiting game specified by equations
 \eqref{eqstartSDENtag}, \eqref{eqlimMacVlaSPDEtag}.
 But $V_{2,lim} \ge V $, where $V$ is the optimal payoff for the first player in the limiting game of
 two players, where the second, measure-valued, player uses $u_{com}$. Consequently,
 \[
 V_1 \ge V_2 \ge V_{2,lim} -\frac{C}{N} \ge V-\frac{C}{N},
 \]
 completing the proof.

\bigskip

{\bf Acknowledgements.} We are grateful to the organizers of the 'Mean field games and related topics 3'
conference in Paris, June 10-12 (2015), for the invitation to present our results, which stimulate our work
on this subject. We thank Peter Caines for presenting our talk on the conference, as the authors turn out to 
be unable to attend the workshop.

\end{document}